\documentclass[11pt]{amsart}
\usepackage[UKenglish]{isodate}
\usepackage{amsmath}
\usepackage{amsthm}
\usepackage{amsbsy}
\usepackage{amssymb}
\usepackage{amsfonts}
\usepackage[dvips]{lscape}
\usepackage{xcolor}
\usepackage{amscd}
\usepackage[all,cmtip]{xy}
\usepackage{euscript}
\usepackage{parskip}
\usepackage{enumerate}
\usepackage{yfonts}
\usepackage{xcolor}
\usepackage{graphicx}
\usepackage{fancyhdr}
\usepackage{tikz-cd}

\usepackage[margin=1in]{geometry}
\usepackage[expansion=false]{microtype}
\usepackage[pdfusetitle,unicode,hidelinks]{hyperref}

\theoremstyle{plain}

\theoremstyle{plain}
\newtheorem{theorem}{Theorem}[section]

\newtheorem{proposition}[theorem]{Proposition}
\newtheorem{lemma}[theorem]{Lemma}

\newtheorem{conjecture}[theorem]{Conjecture}

	\renewcommand{\vec}[1]{\mathbf{#1}}
\theoremstyle{remark}
\newtheorem{remark}[equation]{Remark}

\theoremstyle{definition}
\newtheorem{definition}[theorem]{Definition}

\let\lim=\relax

\DeclareMathOperator*{\lim}{lim}
\newcommand{\ord}{\textup{ord}}
\newcommand{\diag}{\textup{diag}}

\newcommand{\tr}{\textup{tr}}

\renewcommand{\hat}{\widehat}
\newcommand{\Ki}{K_\infty}
\newcommand{\ki}{K_\infty}
\newcommand{\hQ}{\widehat{Q}}

\newcommand{\hR}{\widehat{R}}

\newcommand{\hl}{\widehat l}
\newcommand{\Kl}{\textup{Kl}}

\newcommand{\FF}{\mathbb{F}}
\newcommand{\ZZ}{\mathbb{Z}}

\newcommand{\CC}{\mathbb{C}}

\newcommand{\TT}{\mathbb{T}}
\newcommand{\PGL}{\textup{PGL}}
\newcommand{\PSL}{\textup{PSL}}

\newcommand{\sumstar}{\sideset{}{^*}\sum}

\newcommand{\cO}{\mathcal{O}}

\newcommand{\cG}{\mathcal{G}}
\newcommand{\vol}{\textup{vol}}

\renewcommand{\contentsname}{}
\begin{document}
\title{Ramanujan graphs and exponential sums over function fields}
\author{Naser T. Sardari and Masoud Zargar}

\address{Institute
For Advanced Study,
1 Einstein Drive,
Princeton, New Jersey
08540 USA}
\email{ntalebiz@ias.edu}
\address{Fakult\"at f\"ur Mathematik, Universit\"at Regensburg, Universit\"atsstr.  31, 93040 Regensburg, Germany}
\email{masoud.zargar@ur.de}

\renewcommand{\contentsname}{}

\begin{abstract}
We prove that $q+1$-regular Morgenstern Ramanujan graphs $X^{q,g}$ (depending on $g\in\mathbb{F}_q[t]$) have diameter at most $\left(\frac{4}{3}+\varepsilon\right)\log_{q}|X^{q,g}|+O_{\varepsilon}(1)$ (at least for odd $q$ and irreducible $g$) provided that a twisted Linnik\textendash Selberg conjecture over $\mathbb{F}_q(t)$ is true. This would break the 30 year-old upper bound of $2\log_{q}|X^{q,g}|+O(1)$, a consequence of a well-known upper bound on the diameter of regular Ramanujan graphs proved by Lubotzky, Phillips, and Sarnak using the Ramanujan bound on Fourier coefficients of modular forms. We also unconditionally construct infinite families of Ramanujan graphs that prove that $\frac{4}{3}$ cannot be improved.
\footnotetext{\date\today}
\end{abstract}
\maketitle
\setcounter{tocdepth}{1}
\tableofcontents
\section{Introduction}\label{intro}
We begin by defining  Ramanujan graphs. Suppose $k\geq 3$ is a fixed integer, and let $G$ be a $k$-regular connected graph with adjacency matrix $A_G$. Since $A_G$ is symmetric, all its eigenvalues are real. Furthermore, it is easy to see that $k$ is the largest eigenvalue of the adjacency matrix $A_G$. Letting $\lambda_G$ be the second largest eigenvalue, it is a theorem of Alon\textendash Boppana~\cite{Alon} that $\lambda_G\geq 2\sqrt{k-1}+o(1)$, where $o(1)$ goes to zero as $|G|\rightarrow\infty$. We say that $G$ is a Ramanujan graph if $\lambda_G \leq 2\sqrt{k-1}$. The natural question that arises is if we can construct an infinite sequence of such Ramanujan graphs. Such graphs have been constructed by Lubotzky\textendash Pillips\textendash Sarnak~\cite{Lubotzky1988}, Margulis~\cite{Margulis}, Morgenstern~\cite{Morgenstern}, and others. Though the first two constructions are the same and are $p+1$-regular, $p$ a prime, those of Morgenstern are $q+1$-regular, $q$ any prime power. Recently, Marcus\textendash Spielman\textendash Srivastava have proved the existence of $d$-regular bipartite Ramanujan graphs for arbitrary $d$~\cite{MSS}; Cohen has shown how to construct such $d$-regular graphs in polynomial time~\cite{MCohen}. The study of such graphs is intimately connected to deep questions in number theory, and is also of interest to computer scientists.\\

Lubotzky\textendash Phillips\textendash Sarnark~\cite{Lubotzky1988}, and independently Margulis~\cite{Margulis}, constructed the first examples of Ramanujan graphs; they are Cayley graphs of $\PGL_2(\mathbb{Z}/N\mathbb{Z})$ or $\PSL_2(\mathbb{Z}/N\mathbb{Z})$ with $p+1$ explicit generators, for every prime $p$ and natural number $N$. We denote them by $X^{p,N}.$ The fact that $X^{p,N}$ is a Ramanujan graph follows from the Ramanujan bound on the $p$-th Fourier coefficients of the weight 2 holomorphic modular forms of level $N$, hence the naming of \textit{Ramanujan} graphs. This was proved by Lubotzky, Phillips, and Sarnak in ~\cite{Lubotzky1988}. Another result of Lubotzky\textendash Phillips\textendash Sarnak is that the diameter of every $k$-regular Ramanujan graph $G$ is bounded from above by $2\log_{k-1}|G|+O(1)$. As of today, this is still the best known upper bound on the diameter of a Ramanujan graph. Though it was conjectured by Sarnak that the diameter is bounded from above by $(1+\varepsilon)\log_{k-1}|G|$ as $|G|\to \infty$ (see \cite[Chapter 3]{Peter}), the first author proved that for some infinite families of LPS Ramanujan graphs $X^{p,N}$ the diameter is bigger than $4/3\log_{p}|X^{p,N}|+O(1)$ (see \cite{Sardari2018}).\\
\\
Let $q$ be a prime power and $\mathbb{F}_q$ be the finite field with $q$ elements.   Morgenstern also constructed Ramanujan graphs $X^{q,g}$  by considering a suitable quaternion algebra over $\mathbb{F}_q[t]$, where  $g\in \mathbb{F}_q[t]$ and $\gcd(g,t(t-1))=1$~\cite{Morgenstern}. For a discussion of the connection to strong approximation, see the introduction to the authors' paper \cite{SZ2019}. The graph-theoretic conjecture with which this paper is concerned is the following.
\begin{conjecture}\label{conj:graphs}The diameter of $q+1$-regular Morgenstern Ramanujan graphs $X^{q,g}$ is bounded from above by
\[\left(\frac{4}{3}+\varepsilon\right)\log_{q}|X^{q,g}|+O_{\varepsilon}(1),\]
at least when $q$ is odd and $g$ is irreducible.
\end{conjecture}
Let us assume from now on that the base field $\mathbb{F}_q$ is of odd characteristic. Consider the following system of equations
\begin{equation}\label{maineq}\begin{cases}F(\vec{x})=f,\\ \vec{x}\equiv\boldsymbol{\lambda}\bmod g,\end{cases}\end{equation}
where $F$ is a quadratic form in $4$ variables over $\mathbb{F}_q[t]$, $f,g\in\mathbb{F}_q[t]$, and $\boldsymbol{\lambda}\in\mathbb{F}_q[t]^4$. We say that all the local conditions for the system~\eqref{maineq} are satisfied if the system~\eqref{maineq} has solutions over $\mathbb{F}_q[t]/\langle gh\rangle$ for any nonzero $h\in \mathbb{F}_q[t],$ in addition to $F(\vec{x})=f$ having a solution over $\mathbb{F}_q(\!(1/t)\!)$. As discussed in the introduction to the authors' paper \cite{SZ2019}, the above conjecture would follow from the following very general conjecture regarding strong approximation for quadratic forms over $\mathbb{F}_q[t]$ in $4$ variables. Throughout this paper, the ideal $(a,b)$ of $\mathbb{F}_q[t]$ generated by polynomials $a,b\in\mathbb{F}_q[t]$ (which is a principal ideal domain) will be identified with its monic generator.
\begin{conjecture}\label{mainconj}
Let $F$ be a  quadratic form over $\mathbb{F}_q[t]$ in $4$ variables and of discriminant $\Delta\neq 0$. Let $f, g\in\mathbb{F}_q[t]$ be nonzero polynomials such that $(f\Delta,g)=1$, and let $\boldsymbol{\lambda}\in\mathbb{F}_q[t]^4$ be a quadruple of polynomials at least one of whose coordinates is relatively prime to $g$. Finally, suppose that 
all local conditions for the system~\eqref{maineq}
are satisfied.  There is a solution $\vec{x} \in \mathbb{F}_q[t]^4 $ to \eqref{maineq} if $\deg f \geq (4+\varepsilon)\deg g +O_{\varepsilon,F}(1)$.
\end{conjecture}
In fact, to obtain Conjecture~\ref{conj:graphs}, it suffices to prove the above strong approximation result for Morgenstern quadratic forms given by
\[F(x_1,x_2,x_3,x_4)=\eta_1x_1^2+\eta_2x_2^2+ \eta_3x_3^2+\eta_4x_4^2,\]
where $\eta_1=1,$ $\eta_2=-\nu,$ $\eta_3=-(t-1),$  $\eta_4=\nu(t-1)$ and  $\nu\in\mathbb{F}_q$ is not a square.  In~\cite{SZ2019}, we proved Conjecture~\ref{mainconj} when the number of variables is $d\geq 5$. Furthermore, we showed that for quadratic forms in $d=4$ variables, Conjecture~\ref{mainconj} holds  if we strengthen the condition to $\deg f\geq (6+\varepsilon)\deg g+O_{\varepsilon,F}(1)$ (see~\cite{SZ2019} for details). However, as we saw in Corollary 1.6 of \textit{loc.cit.} this implies the weaker upper bound $(2+\varepsilon)\log_{k-1}|X^{q,g}|+O_{\varepsilon}(1)$ on the diameter of such graphs. The main purpose of this paper is to show that the $\frac{4}{3}$-bound is a consequence of a \textit{twisted} Linnik\textendash Selberg conjecture over $\mathbb{F}_q(t)$ that we formulate; the way we show this is that the twisted Linnik\textendash Selberg conjecture below implies Conjecture~\ref{mainconj} for the Morgenstern quadratic form $F$.\\
\\
Working over function fields is not merely a curiosity; as we will see, in the function field case, the oscillatory integrals, once a suitable weight function is chosen, can be computed explicitly for Morgenstern quadratic forms in terms of Kloosterman sums at the infinite place. This uses the function field stationary phase theorem proved by the authors in \cite{SZ2019}. It turns out that the exponential sums can also be written in terms of Kloosterman sums at the finite places. Furthermore, these graphs constructed by Morgenstern provide us with $q+1$-regular Ramanujan graphs, where $q$ need not be a prime, contrary to the Ramanujan graphs constructed by Lubotzky\textendash Phillips\textendash Sarnak, and independently by Margulis. Most importantly, since the untwisted Linnik\textendash Selberg conjecture is known to be true over function fields~\cite{Cogdell} (a consequence of the Ramanujan conjecture over function fields proved by Drinfeld), we are hopeful that we will be able to prove at least a variant of the twisted Linnik\textendash Selberg conjecture and at least improve upon the best known upper bound on the diameter of such graphs; see Remark~\ref{rem4}. When working over the integers, even if one reduces the optimal strong approximation theorem in four variables to a twisted Linnik\textendash Selberg over the integers, there is little hope in saying much about LPS Ramanujan graphs. Indeed, in the integer case, the Linnik\textendash Selberg conjecture is, as of the writing of this paper, open and is not a consequence of the Ramanujan conjecture. We now devote some time to precisely formulating this conjecture.\\
\\
In order to formulate the twisted Linnik\textendash Selberg conjecture over function fields, we first define the Kloosterman sums under consideration in this paper. In Subsection~\ref{s:add-characters}, we define a nontrivial additive character $\psi=\psi_{\infty}$ on $K_{\infty}:=\mathbb{F}_q(\!(1/t)\!)$ that is trivial  on $\mathcal{O}:=\mathbb{F}_q[t].$ 
Given nonzero $r\in\mathbb{F}_q[t]$, we have an additive character
\[\psi_r:\mathbb{F}_q[t]/(r)\rightarrow\mathbb{C}^*\]
given by sending $x\mapsto\psi\left(\frac{x}{r}\right)$. We may extend this to an additive character on the additive structure of the subring $\mathcal{O}_r$ of $K$ with denominator relatively prime to $r$. The extension we are thinking of here is given by sending $x\in\mathcal{O}_r$ to $\psi\left(\frac{x\bmod r}{r}\right)$.
\begin{definition}Suppose $r\in\mathbb{F}_q[t]$ is nonzero, and suppose $m,n\in\mathcal{O}_r$. Then we define the Kloosterman sum associated to $r,m,n$ as follows: 
\[\Kl_r(m,n):=\sum_{x\in(\mathbb{F}_q[t]/(r))^*}\psi_r\left(mx+n\overline{x}\right),\]
where $\overline{x}$ is the multiplicative inverse of $x$ in $\mathbb{F}_q[t]/(r)$. At the infinite place, we have the following definition of the Kloosterman sum. For $\alpha \in K_{\infty},$  define 
\[\Kl_{\infty}(\psi,\alpha):= \begin{cases}\int_{|x|_{\infty}=\hl} \psi\left(\frac{\alpha}{ x}+x\right) dx,  &\text{ if }  |\alpha|_{\infty}=\hl^2 \text{ for some }l\in \mathbb{Z}\\
 0 &\text{ otherwise,}
 \end{cases}\]
where we are integrating on a subset of $K_{\infty}$ equipped with the Haar measure normalized such that the unit open ball $\TT$ in $K_{\infty}$ has measure $1$, $|.|_{\infty}$ is the norm on $K_{\infty}$ induced by the norm $|a/b|_{\infty}=q^{\deg a-\deg b}$ on $K$, and $\hat{l}:=q^{l}$ throughout this paper.
\end{definition}
By Weil's estimate on the Kloosterman sums, we have square-root cancellation on the Kloosterman sums. The following is a twisted analogue of the Linnik\textendash Selberg conjecture positing that sums of Kloosterman sums exhibit an additional square-root cancellation.
 \begin{conjecture}[Twisted Linnik\textendash Selberg conjecture over function fields]\label{tslconj}Suppose $g\in\mathbb{F}_q[t]$ is a nonzero polynomial, and let $\delta\in\mathbb{F}_q[t]$ be relatively prime to $g$. Then for each integer $T\geq 0$, $\alpha\in\mathbb{F}_q[t]$, and nonzero $a,b\in\mathbb{F}_q[t,g^{-1}]$,
\[\left|\sum_{\substack{|r|=\hat{T}\\ (g,r)=1,\ \delta|r}}\psi_{g^2}\left(\alpha r^{-1}\right)\Kl_r(a,b)\right|\ll_{\varepsilon,\delta}|gab|_{\infty}^{\varepsilon}\hat{T}^{1+\varepsilon}\]
for every $\varepsilon>0$. Furthermore, for every $\varepsilon>0$,
\[\left|\sum_{\substack{|r|=\hat{T}\\ (g,r)=1,\ \delta|r}}\psi_{g^2}\left(\alpha r^{-1}\right)\Kl_r(a,b)\Kl_{\infty}(\psi,ab/r^2)\right|\ll_{\varepsilon,\delta}|gab|_{\infty}^{\varepsilon}\hat{T}^{1+\varepsilon}.\]
\end{conjecture}
\begin{remark}
Note that in the above conjecture, it is equivalent to prove the statement by replacing $|r|=\hat{T}$ with $|r|\leq\hat{T}$. Indeed, if we have the latter for every $T$, then we have the former by simply subtracting the terms for $|r|\leq\hat{T}$ from those with $|r|\leq\frac{\hat{T}}{q}$. Conversely, if we assume the former for every $T$, then the latter follows from the triangle inequality and the fact that we are summing $T\ll_{\varepsilon}\hat{T}^{\varepsilon}$ elements.
\end{remark}
\begin{remark}
As previously mentioned, we remark that the untwisted version of the above conjecture is known to be true. For example, see \cite{Cogdell} for a proof of the Linnik\textendash Selberg conjecture over function fields. The authors hope to study this twisted Linnik\textendash Selberg conjecture in the future.
\end{remark}
We prove the following in Section~\ref{morgenstern}.
\begin{theorem}
Conjecture~\ref{tslconj} implies Conjecture~\ref{mainconj} for the Morgenstern quadratic form $F$ at least when $g$ is irreducible. In particular, the twisted Linnik\textendash Selberg conjecture implies that the Morgenstern Ramanujan graphs have diameter at most
\[\left(\frac{4}{3}+\varepsilon\right)\log_{q}|X^{q,g}|+O_{\varepsilon}(1)\]
if the Ramanujan graph is constructed (at least) for irreducible $g$ over $\mathbb{F}_q$ of odd characteristic.
\end{theorem}
We briefly discuss the history behind the proof. Versions of the circle method, as developed over the integers by Heath-Brown~\cite{Brown} and later over function fields by Browning and Vishe~\cite{Browning}, were successfully applied by the first author to prove optimal strong approximation results for quadratic forms in at least five variables over the integers~\cite{Naser2} and later joint with the second author over $\mathbb{F}_q[t]$~\cite{SZ2019}. These results were achieved with suitable choices of weight functions and delicate divisions of the circles in the two settings. A novelty of the paper ~\cite{SZ2019} was the development of a function field version of the stationary phase theorem that was essential for our calculations there (and here in this paper). That being said, in the case of quadratic forms in four variables (both over the integers and over $\mathbb{F}_q[t]$), which is the case of interest for applications to the covering exponent of $S^3$ and Ramanujan graphs, the results there are suboptimal. It had been observed in \cite[Remark 6.8]{Naser2} that if a certain cancellation in a sum involving exponential sums and oscillatory integrals is true, then optimal strong approximation in the case of four variables would also follow. It was the insight of Browning\textendash Kumarasvamy\textendash Steiner~\cite{BKS} that for the problem regarding optimal covering exponents of $S^3$\textemdash where the quadratic form is a sum of four squares\textemdash the stationary phase theorem can be used to reduce the optimal covering exponent to a twisted Linnik\textendash Selberg conjecture over the integers. We also refer the reader to the work of Steiner~\cite{Steiner} on this integral twisted Linnik\textendash Selberg conjecture; though that paper lead to some insights regarding the limitations of the Kuznetsov trace formula for this covering exponent problem, the desired results were not obtained, unfortunately. In the case of Morgenstern Ramanujan graphs, there are two main differences with the case of the covering exponent problem. Firstly, as mentioned above, the stationary phase theorem in the setting of function fields had to be developed by the authors in~\cite{SZ2019} and applied successfully to the computation of the oscillatory integrals in this paper. Secondly, in addition to the infinite place, the finite places also play a role, complicating the precise computation of the exponential sums under consideration. It is in the Morgenstern case where we are able to relate the exponential sums and oscillatory integrals to Kloosterman sums; for a general quadratic form in four variables, the computations are much more complicated and cannot be written simply in terms of Kloosterman sums. Furthermore, most of the main computations in ~\cite{SZ2019} were done using techniques different from those over the integers; those techniques over function fields are also used in this paper.
\begin{remark}\label{rem4}
Though the twisted Linnik\textendash Selberg Conjecture~\ref{tslconj} would prove the desired Conjecture~\ref{mainconj}, in Conjecture~\ref{tslconj} any power $|g|^{\theta+\varepsilon}$ with $0\leq\theta<1/2$ would allow us to weaken the $\deg f\geq (6+\varepsilon)\deg g+O_{\varepsilon}(1)$ condition, as will be clear from the proof of the theorem above in Section~\ref{morgenstern}. In turn, this would allow us to decrease the upper bound given for the Morgenstern Ramanujan graphs.
\end{remark}
We also prove the following theorem stating that the coefficient $\frac{4}{3}$ cannot be improved upon; see Section~\ref{lowerbound} for a proof.
\begin{theorem}\label{lowerdiam}
Suppose that $q \equiv 3$ mod 4.
There exist infinitely many  $g\in \mathbb{F}_q[t]$ such that the Morgenstern Ramanujan graph $X^{q,g}$ is non-bipartite (or bipartite) and  
\[\textup{diam}(X^{q,g}) \geq \frac{4}{3}\log_{q}|X^{q,g}|+O(1),\]
where $O(1)$ is an absolute constant. 
\end{theorem}
This also gives us a new family of $q+1$-regular non-bipartite (or bipartite) Ramanujan graphs with large diameter and with $q\equiv 3$ mod 4 any prime power.
\section{Lower bound on the diameter}\label{lowerbound}
In this section, we prove Theorem~\ref{lowerdiam}. Our argument is similar to the previous argument  of the first author in~\cite[Theorem 1.2]{Sardari2018}.
\begin{proof}
Since $q \equiv 3$ mod 4, $-1$ is a quadratic non-residue in $\mathbb{F}_q,$ and  the Morgenstern quadratic form for $\nu=-1$ is
\[
F(x_1,x_2,x_3,x_4):=x_1^2+x_2^2-(t-1)(x_3^2+x_4^2).
\]
Let $g(t)\in \mathbb{F}_q[t]$ be any irreducible polynomial relatively prime to $t(t-1)$ such that $t$ is a quadratic non-residue in the finite field $\mathbb{F}_{q^{\deg(g)}}:=\mathbb{F}[t]/\langle g\rangle,$ and $-1$ is a quadratic residue in $\mathbb{F}_{q^{deg(g)}},$ (which means $\deg(g)$ is even). Then it follows from the work of Morgenstern that $X^{q,g}$ is isomorphic to  the Cayley graph of $\PGL_2(\mathbb{F}_{q^{\deg(g)}})$ generated by $q+1$ generators. The identification is given by the the following map sending the quaternion
\[
x_1+ix_2+jx_3+kx_4 \mapsto \begin{bmatrix}x_1-x_2i & x_3-x_4i \\ (t-1)(x_3+x_4i)& x_1+x_2i\end{bmatrix},
\]
where $i$ is a choice of $\sqrt{-1}$ in $\mathbb{F}_{q^{deg(g)}}.$ Moreover the quadratic residue of the determinant gives a bipartite structure on $X^{q,g}.$ Let $I:=\begin{bmatrix} 1 &0 \\ 0 & 1\end{bmatrix}\in X^{q,g}$ and $W:=\begin{bmatrix} 1 &0 \\ 0 & -1\end{bmatrix}\in X^{q,g}.$ We show that 
\[
\text{dist}(I,W) \geq \frac{4}{3}\log_{q}|X^{q,g}|+O(1).
\]
Suppose that there exists a path of minimal length $h$ from $I$ to $W.$ Then, by \cite[Theorem 4.5]{Morgenstern} there exists an integral solution $\vec{u}:=(u_1, \dots, u_4) \in \mathbb{F}_q[t]^4 $ to 
\[
x_1^2+x_2^2-(t-1)(x_3^2+x_4^2)=t^h,
\]
where $g|\gcd(u_1,u_3,u_4),$  $\gcd(g,u_2)=1$ and $t-1|\gcd(u_1-1,u_2)$. This implies that $u_2^2 \equiv t^h$ mod $g^2.$ Here, we are using that $g$ is irreducible and relatively prime to $t$, and so $g$ cannot divide both $\pm t^h$. Since $t$ is a quadratic non-residue mod $g$ and $t^h$ is a square mod $g$, $h$ is even. Suppose that $h=2l.$ If $l \geq 2 \deg(g)$ then 
\[
\text{dist}(I,W)=h\geq  4 \deg(g) = \frac{4}{3}\log_{q}|X^{q,g}|+O(1).
\]
So, we suppose that $l<2 \deg(g).$ We have 
\(u_2^2 \equiv t^h \equiv t^{2l}\mod g^2.\) This implies 
\(
u_2\equiv \pm t^l \text{ mod } g^2.
\) We write $u_2=ag^2\pm t^l$ for some $a\in \mathbb{F}_q[t].$ For $a\neq 0,$  $\deg(ag^2\pm t^l)\geq 2\deg(g)$, and we have 
\[
2l=\deg(F(\vec{u}))\geq 2 \deg(u_2)= 2\deg(ag^2\pm t^l) \geq 4 \deg(g)> 2l,
\]
because  $-1$ is a non-square residue modulo $q.$
This is a contradiction. So $a=0,$ and this implies $u_1=0,$ which is a contradiction since $t-1|u_1-1.$ This proves Theorem~\ref{lowerdiam} in the bipartite cases.
\\

Next, we give an infinite family of non-bipartite Morgenstern Ramanujan graphs with the same lower bound on their diameter. Let $r\in \mathbb{F}_q[t]$ be any irreducible polynomial relatively prime to $t(t-1)$ and such that $t$ and $-1$ are quadratic residues in the finite field $\mathbb{F}_{q^{\deg(r)}}:=\mathbb{F}[t]/\langle r\rangle.$ Consider the Morgenstern Ramanujan graph $X^{q,(t^2+1/4)r}.$ It follows from the work of Morgenstern that $X^{q,(t^2+1/4)r}$ is isomorphic to  the Cayley graph of $\PSL_2\left(\mathbb{F}[t]/\langle (t^2+1/4)r\rangle\right)$ generated by $q+1$ generators, and $X^{q,(t^2+1/4)r}$ is a non-bipartite Ramanujan graph. Let $I$ and $W$ be as before and define $I^{\prime}:=\begin{bmatrix} 1 &r \\ 0 & 1\end{bmatrix}.$ Let $\sqrt{-1}$ be a square root of $-1$ in $\mathbb{F}[t]/\langle (t^2+1/4)r\rangle.$ Also let $W^{\prime}:=\sqrt{-1}W\in \PSL_2.$ We show that 
\[
\max(\text{dist}(I, I^{\prime}) ,\text{dist}(I,W^{\prime}))  \geq \frac{4}{3}\log_{q}|X^{q,(t^2+1/4)r}|+O(1).
\]
Assume to the contrary that $\max(\text{dist}(I, I^{\prime}) ,\text{dist}(I,W^{\prime})) < 4 \deg(r)=\frac{4}{3}\log_{q}|X^{q,(t^2+1/4)r}|+O(1).$ Since $\text{dist}(I, I^{\prime})< 4 \deg(r),$ by \cite[Theorem 4.5]{Morgenstern} it follows that there exists an  integral solution
 \[
 a_1^2+a_2^2-(t-1)(a_3^2+a_4^2)=t^{h_1}
 \]
 for some $h_1< 4 \deg (r)$, where $r|\gcd (a_2,a_3,a_4)$, $(t-1)|\gcd(a_1-1,a_2)$, and at least one of $a_3$ or $a_4$ is non-zero.  This implies $a_1^2\equiv t^{h_1}$ mod $r^2.$ We consider two cases: $h_1$ is even or $h_1$ is odd. First, suppose that $h_1=2l_1.$ Then $a_1\equiv \pm t^{l_1}$ mod $r^2$. This implies $a_1=c_1g^2\pm t^{l_1}$ for some $c_1\in \mathbb{F}_q[t].$
Suppose that $c_1\neq 0.$ Then $\deg(a_1)\geq 2\deg(r),$ and we have 
\[
h_1=\deg( a_1^2+a_2^2-(t-1)(a_3^2+a_4^2))\geq 2\deg(a_1) \geq 4 \deg(r). 
\]
This is a contradition. Hence, $c_1=0.$ This implies $a_3=a_4=0,$ which is also a contradiction. Therefore, $h_1=2l_1+1$ is odd. 
\\

Similarly, since $\text{dist}(I,W^{\prime})<4 \deg(r)$ by assumption, it follows that there exists an integral solution
 \[
 b_1^2+b_2^2-(t-1)(b_3^2+b_4^2)=t^{h_2}
 \]
 for some $h_2< 4 \deg (r)$, where $r|\gcd (b_1,b_3,b_4)$, $(t-1)|\gcd(b_1-1, b_2)$. By a similar argument it follows that $h_2=2l_2+1$ is odd and we have $b_2^2\equiv t^{h_2}$ mod $r^2.$ Therefore, we have
 \[
\begin{cases}
a_1\equiv 1 \text{ mod } (t-1), \text{ and } a_1^2 \equiv t^{2l_1+1} \text{ mod } r^2, \text{ and } l_1<2\deg(r),
\\
b_2\equiv 0 \text{ mod } (t-1), \text{ and } b_2^2 \equiv t^{2l_2+1} \text{ mod } r^2, \text{ and } l_2<2\deg(r).
\end{cases}
 \] 
Without loss of generality, suppose that $l_1\geq l_2.$ Then, we have 
\[
a_1\equiv \pm t^{l_1-l_2} b_2 \text{ mod } r^2.
\]
Note that $\deg(a_1)< l_1+1/2< 2\deg(r)$ and $\deg(t^{l_1-l_2} b_2)<l_1+1/2< 2\deg(r).$ Hence, 
\[
a_1= \pm t^{l_1-l_2} b_2.
\]
This contradicts with $\begin{cases}
a_1\equiv 1 \text{ mod } (t-1),\\
b_2\equiv 0 \text{ mod } (t-1).
\end{cases}
$ This completes the proof of our theorem. 

 \end{proof}

\section{Recollections on the delta method}\label{deltamethod}
The primary purpose of this section is to collect some of the facts related to the delta method over $\mathbb{F}_q[t]$. This section also serves the purpose of setting the notation for the rest of the paper. For details, the reader may consult Section 2 of the authors' paper~\cite{SZ2019}.\\
\\
Roughly, the delta method is a procedure by which one rewrites the delta function over integral points inside a region as a weighted sum of characters. In this section, we define a weighted sum $N(w, \boldsymbol{\lambda})$ counting the number of integral solutions to the system~\eqref{maineq}.
\subsection{Notation}
As in the authors' paper~\cite{SZ2019}, we let $K=\mathbb{F}_q(t)$ and let $\mathcal{O}=\mathbb{F}_q[t]$ be its ring of integers. We denote the prime at infinity $t^{-1}$ by $\infty$. We may equip $K$ with the norm at infinity given by
\[|a/b|_{\infty}:=q^{\deg a-\deg b}.\]
Completing $K$ with respect to this norm gives $K_{\infty}$. Henceforth, we drop the subscript $\infty$ from $|.|_{\infty}$ and write $|.|$ for simplicity. We can also extend the norm to higher dimensions: for every $d$, the natural norm on $K_{\infty}^d$ is given by $|\mathbf{a}|:=\max_i|a_i|$. These norms equip $K_{\infty}^d$ with the metric topology.\\
\\
Note that we may identify $K_\infty$ with the field
$$
\FF_q(\!(1/t)\!)=\left\{\sum_{i\leq N}a_it^i: \mbox{for $a_i\in \FF_q$ and some $N\in\ZZ$} \right\}.
$$
The (open) unit ball in this topological space is
$$\TT=\{\alpha\in K_\infty: |\alpha|<1\}=\left\{\sum_{i\leq -1}a_it^i: \mbox{for $a_i\in \FF_q$}
\right\}.
$$
\subsection{Characters}\label{s:add-characters}
Let $e_q:\FF_q\rightarrow \CC^*$ be the nontrivial additive character given by sending $a\in \FF_q$ to 
$e_q(a)=\exp(2\pi i \tr(a)/p)$, where $p:=\textup{char }\mathbb{F}_q$ and $\tr: \FF_q\rightarrow \FF_p$ is the
trace map. From this, we obtain the non-trivial additive character $\psi:
K_\infty\rightarrow \CC^*$ given by $\psi(\alpha)=e_q(a_{-1})$ for any 
$\alpha=\sum_{i\leq N}a_i t^i$ in $\Ki$. By construction, $\psi|_\cO$ is trivial. 
Furthermore, for any $\gamma \in \Ki$, the map $\alpha\mapsto \psi(\alpha\gamma)$ is also an additive
character on $\Ki$. A basic lemma that will be useful in our computations is the following.
\begin{lemma}[Kubota, Lemma 7 of
\cite{kubota}]\label{lem:orthogsum}
$$
\sum_{\substack{b\in \cO\\ |b|<\hat N}}\psi(\gamma b)=\begin{cases}
\hat N, & \mbox{if $|(\!(\gamma)\!)|<\hat N^{-1}$,}\\
0, & \mbox{otherwise},
\end{cases}
$$
for any $\gamma \in \ki$ and any integer $N\geq 0$, where $(\!(\gamma)\!)$ is the part of $\gamma$ with all degrees negative.
\end{lemma}
We also have the following
\begin{lemma}[Kubota, Lemma 1(f) of \cite{kubota}]\label{lem:orthog}
Let $Y\in \ZZ$ and $\gamma\in  \ki$. Then 
$$
\int_{|\alpha|<\hat{Y}} \psi(\alpha \gamma) d \alpha=\begin{cases}
\hat Y, &\mbox{if $|\gamma|<\hat Y^{-1}$},\\
0, &\mbox{otherwise.}
\end{cases}
$$
\end{lemma}
In particular, if we set $Y=0$, then we obtain the following expression for the delta function on $\mathcal{O}$:
\[\delta(x)=\int_{\TT}\psi(\alpha x)d\alpha,\]
where 
\(
\delta(x)=\begin{cases} 1 &\text{ if } x=0,\\ 0 &\text{ otherwise.}  \end{cases}
\)
\subsection{The delta function}The idea now is to decompose $\TT$ into a disjoint union of balls (with no minor arcs) which is the analogue of Kloosterman's version of the circle method in this function field setting. This is done via the following lemma of Browning and Vishe \cite[Lemma~4.2]{Browning}.
\begin{lemma}\label{lem:dissection}
For any $Q>1$ we have a disjoint union 
\[
\TT=\bigsqcup_{\substack{r\in \cO\\ |r|\leq \hat Q\\
\textup{$r$ monic}}}
\bigsqcup_{\substack{a\in \cO\\ |a|<|r|\\ (a,r)=1} }
\left\{\alpha\in \TT: |r\alpha-a|<\hat Q^{-1}\right\}.\]
\end{lemma}
The following follows from Lemma \ref{lem:dissection}; see \cite[Lemma~2.4]{SZ2019}.
\begin{lemma}\label{delta} Let $Q\geq  1$ and $n\in \cO.$ 
We have 
\begin{equation}\label{deltaeq}
\delta(n)= \frac{1}{\hQ^{2}}
\sum_{\substack{
r\in \cO\\
|r|\leq \hat Q\\
\textup{$r$ monic}
} }
\sumstar_{\substack{
|a|<|r|} }
\psi\left(\frac{an}{r}\right) h\left(\frac{r}{t^Q},\frac{n}{t^{2Q}}\right),
\end{equation}
where we henceforth put
\(
\sumstar_{\substack{
|a|<|r| }}
:=\sum_{\substack{
a\in \cO\\
a\textup{ monic}\\ 
|a|<|r|\\ (a,r)=1} },
\)
and $h$ is only  defined for $x\neq 0$ as:
\[h(x,y)=\begin{cases}|x|^{-1} & \mbox{if $|y|<|x|$}\\ 0 & \mbox{otherwise.}\end{cases}\]
Moreover, 
\[
\frac{1}{\hQ^{2}}h\left(\frac{r}{t^Q},\frac{n}{t^{2Q}}\right)= \int_{|\alpha|<|r|^{-1}\hQ^{-1}} \psi\left(\alpha n\right) d \alpha.
\label{delta}\]
\end{lemma}
\subsection{Smooth sum $N(w,\boldsymbol{\lambda})$}
Let
\[
w(\vec{x})=\begin{cases}
1 \text{ if } |\vec{x}| \leq |f|^{1/2},
\\
0 \text{ otherwise.}
\end{cases}
\]
Note that
\[
w(g\vec{t}+\boldsymbol{\lambda})=\begin{cases}
1 \text{ if } |\vec{t}|  <  \hR,
\\
0 \text{ otherwise,}
\end{cases}
\]
where  $R:= \lfloor{\deg(f)/2} -\deg(g)+1\rfloor.$
Assume that $\vec{x}\in \mathcal{O}^d$ satisfies the conditions $F(\vec{x})=f$ and $\vec{x}\equiv \boldsymbol{\lambda}\bmod g$. We uniquely write $\vec{x}=g\vec{t}+ \boldsymbol{\lambda},$ where $\vec{t}\in  \mathcal{O}^d$ and $\boldsymbol{\lambda}=(\lambda_1,\dots,\lambda_d)$ for $\lambda_i$ of degree strictly less than that of $g$. Define 
\begin{equation}\label{defk}
k:= \frac{f-F(\boldsymbol{\lambda})}{g}.
\end{equation}
If
$F(\vec{x}) =f$, then $g^2F(\vec{t})+2g\boldsymbol{\lambda}^T A\vec{t} =f-F(\boldsymbol{\lambda})$ which implies that $g|2\boldsymbol{\lambda}^TA\vec{t}-k.$ 
 Then, 
$F(\vec{t})+\frac{1}{g}(2\boldsymbol{\lambda}^T A\vec{t}-k )=0.$
We also define 
\[G(\vec{t}):=\frac{F(g\vec{t}+\boldsymbol{\lambda})-f}{g^2}=F(\vec{t})+\frac{1}{g}(2\boldsymbol{\lambda}^T A\vec{t}-k).\]
Finally, we  define $$N(w,\boldsymbol{\lambda}):=\sum_{\vec{t}} w(g\vec{t}+\boldsymbol{\lambda}) \delta{(G(\vec{t}))},$$  where $\vec{t}\in \mathcal{O}^{d}$. Note that $N(w,\boldsymbol{\lambda})$ is the weighted number of $\vec{x}\in\mathcal{O}^d$ satisfying the conditions $F(\vec{x})=f$ and $\vec{x}\equiv \boldsymbol{\lambda}\bmod g$. We apply the delta expansion in \eqref{deltaeq} to   $ \delta{(G(\vec{t}))}$ and follow the computations in \cite[Section 2.4]{SZ2019}, and obtain 
\begin{equation}\label{newequu}N(w,\boldsymbol{\lambda})=\frac{1}{|g|\hQ^2}\sum_{\substack{
r\in \cO\\
|r|\leq \hat Q\\
\text{$r$ monic}
} }\sum_{\vec{c}\in\mathcal{O}^d}|gr|^{-d}S_{g,r}(\vec{c})I_{g,r}(\vec{c}),\end{equation}
where $I_{g,r}(\vec{c})$ and $S_{g,r}(\vec{c})$ are defined by 
\begin{equation}\label{intt}
I_{g,r}(\vec{c}):= \int_{\Ki^d} h\left(\frac{r}{t^Q},\frac{G(\vec{t})}{t^{2Q}}\right) w(g\vec{t}+\boldsymbol{\lambda}) \psi\left(\frac{\left< \vec{c},\vec{t} \right>}{gr}    \right) d\vec{t},
\end{equation}
and
\begin{equation}\label{sala}S_{g,r}(\vec{c}):=\sum_{\substack{\ell\in \cO\\ |\ell|<|g|}}\sumstar_{|a|<|r|} S_{g,r}(a,\ell,\vec{c})\end{equation}
with
\begin{equation}\label{salam}
S_{g,r}(a,\ell,\vec{c}):=\sum_{\vec{b}\in (\mathcal{O}/(gr))^d} \psi\left(\frac{(a+r\ell)(2\boldsymbol{\lambda}^T A\vec{b}-k) + ag F(\vec{b})-\left< \vec{c},\vec{b} \right>} {gr}\right).
\end{equation}
We henceforth assume that \textit{$r$ is always monic} without saying so. In the next two sections, we give explicit formulas for $S_{g,r}$ and $I_{g,r}$ when our quadratic form is the Morgenstern quadratic form.

\section{The oscillatory integrals $I_{g,r}(\vec{c})$}\label{osil}
In this section, we give explicit formulas for the oscillatory integrals $I_{g,r}(\vec{c})$ in terms of the  Kloosterman sums at infinity. Suppose that $F(\vec{t}):=\sum_{i} \eta_i t_i^2$ is the   Morgenstern quadratic form and $F^*(\vec{c}):= \sum_{i}  c_i^2/\eta_i$ is its dual.  Let  $\kappa:=\max_{i} |\frac{c_i}{g }|.$  In this section, we assume that $Q= R+1.$
\begin{proposition}\label{prop:oscmorgen}
Suppose that $  |r| \leq \hQ.$ For the Morgenstern quadratic form, we have 
\[
I_{g,r}(\vec{c})=\begin{cases}
0, &\text{ if } \kappa \geq \hQ/\hR,
\\
0, &  \kappa= \frac{|r|}{\hR}, \deg(f) \text{is even},  \text{and } \hQ q^{-3} < |r| \leq \hQ,
\\
I_{g,r}(\vec{0}), &\text{ if } \kappa< \frac{|r|}{\hR},
\\  
I_{g,r}(\vec{0}), &\text{ if } \kappa= \frac{|r|}{\hR},  \max{(|c_3|,|c_4|)}> \max{(|c_1|,|c_2|)}, 
\\ &\text{ and }\deg(f) \text{is even},  |r| \leq \hQ q^{-3},
\\-\hQ^2|g|^{2} |r|^{2}|F^{*}(\vec{c})|^{-1}\Kl_{\infty}\left(\psi, \frac{kF^{*}(\vec{c})}{4r^2g^3}\right),  &\text{otherwise.} \end{cases}
\]
\end{proposition}
We give the proof of Proposition~\ref{prop:oscmorgen} at the end of this section. We proceed by citing some general results from \cite{SZ2019} that are not specific to the Morgenstern quadratic form. In particular, we do not restrict the number of variables to $d=4$ for the moment. Recall that 
\[G(\vec{t}):=\frac{F(g\vec{t}+\boldsymbol{\lambda})-f}{g^2}=F(\vec{t})+\frac{1}{g}(2\boldsymbol{\lambda}^T A\vec{t}-k),\]
where $k=\frac{f-F(\lambda)}{g}.$ 
  We have  
\begin{equation}\label{intt}
I_{g,r}(\vec{c})= \int_{\Ki^d} h\left(\frac{r}{t^Q},\frac{G(\vec{t})}{t^{2Q}}\right) w(g\vec{t}+\boldsymbol{\lambda}) \psi\left(\frac{\left< \vec{c},\vec{t} \right>}{gr}    \right) d\vec{t}= \int_{\substack{|\vec{t}| < \hR \\ |G(\vec{t})|< \hQ|r|}} \frac{\hQ}{|r|}  \psi\left(\frac{\left< \vec{c},\vec{t} \right>}{gr}    \right) d\vec{t}.
\end{equation}
We cite \cite[Lemma 5.5]{SZ2019}.
\begin{lemma}\label{auxlem}
Let $Q$ and  $R$  be as above, and suppose that $|\vec{t}|< \hR.$ Then $|G(\vec{t})| < \hQ|r|$ is equivalent to $|F(\vec{t})-k/g|<\hQ |r|.$ Moreover, if $|G(\vec{t})|< \hQ|r|$, then $|G(\vec{t}+\boldsymbol{\zeta})|  < \hQ|r|$ for every $\boldsymbol{\zeta}\in K_\infty^d,$ where $|\boldsymbol{\zeta}| \leq \min(|r|,\hR).$
\end{lemma}
We cite \cite[Lemma 5.6]{SZ2019}.
\begin{lemma}\label{o}
Suppose that $ \kappa \geq \hQ/\hR$  and  $  |r| \leq \hQ$. Then, $I_{g,r}(\vec{c})=0.$
\end{lemma}
\begin{lemma}
Suppose that $ \kappa= \frac{|r|}{\hR}, \deg(f) \text{is even},  \text{and } \hQ q^{-3} < |r| \leq \hQ.$ Then, $I_{g,r}(\vec{c})=0.$
\end{lemma}
\begin{proof}
Since $\deg(f)$ is even, $R=\frac{\deg (f)}{2} -\deg(g)+1$ and $Q=\frac{\deg (f)}{2} - \deg(g)+2.$ Hence, $ \hQ|r| > |k/g|$, and  by Lemma~\ref{o},  $|G(\vec{t})| < \hQ|r|$ is equivalent to 
\[
|F(\vec{t})|< \hQ|r|.
\]
Since  $\hQ q^{-3} < |r| \leq \hQ$ and  $|\vec{t}|<\hR,$
\[
|F(\vec{t})| \leq \widehat{2R-1} < \hQ|r|.
\] 
So, the inequality $|G(\vec{t})| < \hQ|r|$  is satisfied automatically, and we have 
\[
I_{g,r}(\vec{c})=  \frac{\hQ}{|r|}  \int_{|\vec{t}| < \hR }   \psi\left(\frac{\left< \vec{c},\vec{t} \right>}{gr} \right)   d\vec{t}=  0,
\]
where we used $ \kappa= \frac{|r|}{\hR}.$ This completes the proof of our lemma. 
\end{proof}

\begin{lemma}\label{noosc}
Suppose that either
\begin{enumerate}
\item $\kappa< \frac{|r|}{\hR},$ \label{1cas}
\item $ \kappa= \frac{|r|}{\hR},  \max{(|c_3|,|c_4|)}>\max{(|c_1|,|c_2|)},\deg(f) \text{is even},  \text{and } |r| \leq \hQ q^{-3} $\label{2cas},
\end{enumerate}
 then  $I_{g,r}(\vec{c})=I_{g,r}(\vec{0}).$
\end{lemma}
\begin{proof}
Suppose \eqref{1cas}. Since $\max_{i}(|c_i|)< \frac{|gr|}{\hR}$ and $|\vec{t}|<\hR$,  $ \psi\left(\frac{\left< \vec{c},\vec{t} \right>}{gr}    \right)=1.$ 
Hence, we have 
\[
I_{g,r}(\vec{c})= \int_{\substack{|\vec{t}| < \hR \\ |G(\vec{t})|< \hQ|r|}} \frac{\hQ}{|r|}   d\vec{t}=  I_{g,r}(0).
\]
Suppose \eqref{2cas} and that $ \hQ|r| \leq |k/g|$.  By lemma~\ref{auxlem},
\[
|F(\vec{t})-k/g|< \hQ|r| \leq |k/g|.
\]
Hence, the top degree of $F(\vec{t})$ and $k/g$ are the same. Since $\deg(f)$ is even,  $\deg(k/g)=\deg(f/g^2)$ is even. The top degree of $F(\vec{t})$ is even as well. Hence, $ \max{(|t_3|,|t_4|)}< \max{(|t_1|,|t_2|)}.$ Hence,
\(
| \left< \vec{c},\vec{t} \right>|  < |c||\vec{t}|, 
\)
and 
\[
 \psi\left(\frac{\left< \vec{c},\vec{t} \right>}{gr}    \right)=1,
\]
which implies $I_{g,r}(\vec{c})=I_{g,r}(\vec{0}).$ Finally, suppose  \eqref{2cas} and that $ \hQ|r| > |k/g|$. This implies that $ |r| = \hQ q^{-3}.$ By Lemma~\ref{o},  $|G(\vec{t})| < \hQ|r|$ is equivalent to 
\(
|F(\vec{t})|< \hQ|r|=\hR^2q^{-1}.
\)
We have  \[2\deg(\max{(|t_3|,|t_4|)})+1 \leq \deg(F(\vec{t})).\] Hence,
\[
\deg(\max{(|t_3|,|t_4|)}) \leq  \frac{\deg(F(\vec{t}))-1}{2} \leq  R-2.
\]
Therefore, 
\(
 \psi\left(\frac{\left< \vec{c},\vec{t} \right>}{gr}    \right)=1.
\)
This implies $I_{g,r}(\vec{c})=I_{g,r}(\vec{0}).$ 

\end{proof}
\subsection{Stationary phase theorem over function fields}
In \cite[Proposition 4.5]{SZ2019}, we proved a version  of the stationary phase theorem in the function fields setting. We proceed by defining some new notations and cite a special case of \cite[Proposition 4.5]{SZ2019}. Let $h\in K_{\infty}$ and define  
 \begin{equation}\label{epsfact}
\cG(h):=\begin{cases}\min(|h|_{\infty}^{-1/2},1)   &\text{ if } \ord(h)\text{ is } even,
\\
|h|_{\infty}^{-1/2} \varepsilon_{h} &\text{ if } \ord(h)\geq 1\text{ and is } odd,
\\
1 &\text{ otherwise,}
\end{cases}
\end{equation}
where $\varepsilon_{h}:=\frac{G(h)}{|G(h)|}$ and  $G(h):=\sum_{x\in \mathbb{F}_q} e_q(a_h x^2)$ is the gauss sum associated to $a_h$ the top degree coefficient of $h.$ we cite~\cite[Lemma 4.6]{SZ2019}.
\begin{lemma}\label{ggg} For every $f\in K_{\infty},$
we have 
\[
\int_{\TT} \psi(fu^2) du=\cG(f).
\]
\end{lemma}
For $\alpha\in K$ and $a\in \mathbb{Z}$,  define
\[
B_{\infty}(\psi,a,\alpha):=\int_{|x|_{\infty}=\hat{a}} \psi(\frac{\alpha}{ x}+x) dx.
\]
We write  $\alpha= t^{2a+b}\alpha^{\prime} (1+\tilde{\alpha})$ and  $x=t^ax^{\prime}(1+\tilde{x})$ for unique $\tilde{\alpha},\tilde{x}\in \TT $ and $\alpha^{\prime}, x^{\prime} \in \FF_q.$  Note that for $b=0,$ we have   $B_{\infty}(\psi,a,\alpha)=\Kl_{\infty}(\psi,\alpha)$. Let $\Kl(\alpha,\FF_q):=\sum_{x\in \FF_q^*}  e_q\left(\frac{\alpha}{ x}+x\right).$ We cite \cite[Lemma 5.8]{SZ2019}.
\begin{lemma}\label{Bessel}

We have 
\[
B_{\infty}(\psi,a,\alpha)=\begin{cases} 
(q-1)\hat{a} &\text{ if } \max(a+b,a)<-1, \text{ and } b\neq 0,
 \\
-\hat{a} &\text{ if } \max(a+b,a)=-1, \text{ and } b\neq 0,
\\
0   &\text{ if } \max(a+b,a)>-1, \text{ and } b\neq 0. \end{cases}
\]
\[
\Kl_{\infty}(\psi,\alpha)=\begin{cases} 
(q-1)\hat{a} &\text{ if } a<-1,
\\
\hat{a} \Kl(\alpha^{\prime},\FF_q) &\text{ if } a=-1,
\\
\hat{a}\sum_{{x^{\prime}}^2= \alpha^{\prime}}  \psi\Big(2t^{a}x^{\prime}(1+\tilde{\alpha})^{1/2} \Big) \cG(2x^{\prime}t^a)   &\text{ if } \alpha^{\prime} \text{ is a quadratic residue,}
\\
0   &\text{ if } \alpha^{\prime} \text{ is not a quadratic residue.} \end{cases}
\]
\end{lemma}
\begin{proof}[Proof of Proposition~\ref{prop:oscmorgen}]
In this proof, we assume that we are working with the Morgenstern quadratic form. By Lemma~\ref{o}, we have $I_{g,r}(\vec{c})=0$ for $ \kappa \geq \hQ/\hR$  and  $  |r| \leq \hQ$. By assuming the conditions of Lemma~\ref{noosc}, it follows that 
\(
I_{g,r}(\vec{c})=I_{g,r}(\vec{0}).
\) Hence, the remaining cases  correspond to 
\begin{enumerate}
\item$ \frac{|r|}{\hR} <  \kappa,$ \text{or}
\item $ \kappa= \frac{|r|}{\hR}, \text{ and } \max{(|c_3|,|c_4|)}\leq \max{(|c_1|,|c_2|)}  \textit{ or } \deg(f) \text{ is odd}, $
\end{enumerate}
 and we proceed to conclude the proposition in these two cases.\\
\\
By Lemma~\ref{auxlem},  $|G(\vec{t})| < \hQ|r|$ is equivalent to $|F(\vec{t})-k/g|<\hQ |r|$ for $|\vec{t}|<\hR.$ By Lemma~\ref{lem:orthog}, we have 
$$
\int_{\TT} \psi\Big(\frac{\alpha} {rt^Q } (F(\vec{t})-k/g)\Big) d \alpha=\begin{cases}
 1, &\mbox{if $|F(\vec{t})-k/g|< \hQ|r|$},\\
0, &\mbox{otherwise.}
\end{cases}
$$
We replace the above integral for detecting $|F(\vec{t})-k/g|< \hQ|r|.$ Hence, by \eqref{intt}
\[
I_{g,r}(\vec{c})=\frac{\hQ}{|r|} \int _{\TT}\int_{|\vec{t}| <\hR}\psi\left(\frac{\left< \vec{c},\vec{t} \right>}{gr} +\frac{\alpha} {rt^Q } (F(\vec{t})-k/g)   \right) d\vec{t} d\alpha.\]
Recall  that $F(\vec{t})=\sum_{\eta_i} \eta_it_i^2.$ We have 
 \[
\frac{\left< \vec{c},\vec{t} \right>}{gr} +\frac{\alpha} {rt^Q } (F(\vec{t})-k/g)=\frac{-\alpha k}{rgt^Q}+ \frac{1}{r} \Big(\sum_{i} \frac{c_it_i}{g} +\frac{\alpha  \eta_i t_i^2}{t^Q}   \big),
 \]
 Hence,
 \(
I_{g,r}(\vec{c})=\frac{\hQ}{|r|} \int _{\TT} \psi(\frac{-\alpha k}{rgt^Q})  I_{g,r}(\alpha,\vec{c})d\alpha,\)
where 
\[
I_{g,r}(\alpha,\vec{c}):=\prod_{i=1}^{4}\int_{|t_i|<\hR} \psi \Big( \frac{1}{r} (\sum_{i} \frac{c_it_i}{g} +\frac{\alpha  \eta_i t_i^2}{t^Q}   )\Big) dt_i.
 \]
 The phase function has a critical point at $t_i=\frac{-c_i t^Q}{2g\eta_i \alpha}.$ This critical point is inside the domain $|t_i|<\hR,$  if $ |\alpha |>\kappa_i$, where $\kappa_i:=\frac{|c_i|\hQ}{|g||\eta_i|\hR}$.  Given $\alpha\in\TT$, we partition the indices into: 
\begin{equation*}
\begin{split}
CR:=\left\{1\leq i\leq 4: |\alpha |>\kappa_i \right\},
\\
NCR:=\left\{1\leq i\leq 4: |\alpha | \leq \kappa_i \right\}.
\end{split}
\end{equation*}
For $i\in NCR,$ we change the variables to 
\(
v_i=t_i +\frac{\alpha \eta_i g}{c_it^Q} t_i^2
\). Note that $v_i$ is an analytic map in terms of $t_i$ and $\frac{\partial v_i}{\partial t_i}(0)=1.$ Hence, by~\cite[Proposition~4.2]{SZ2019}, $v_i$ is a bijection from $|t_i|<\hat{R}$ to $v_i<\hat{R}$. 
For $i\in CR$, we change the variables to  $w_i=t_i+\frac{c_it^Q}{2g\eta_i \alpha}.$
   By \cite[Section 4]{SZ2019}, we have 
\[
I_{g,r}(\alpha,\vec{c})= \prod_{i\in NCR}\int_{|v_i|<\hR}\psi\left(\frac{c_iv_i }{gr}  \right) dv_{i}
\times \prod_{i\in CR}\psi (-\frac{t^Q{c_i}^2}{4rg^2\eta_i \alpha}) \int_{|w_i|<\hR}\psi
\left( \frac{\alpha \eta_i }{rt^{Q}}w_i^2  \right) dw_i.
\]
By Lemma~\ref{lem:orthog}, Lemma~\ref{ggg}, we have 
\[
\int_{|v_i|<\hR}\psi\left(\frac{c_iv_i }{gr}  \right) dv_{i}=\begin{cases}
\hR, &\mbox{if $\frac{|c_i|}{|g|}<\frac{|r|}{\hR}$},\\
0, &\mbox{otherwise,}
\end{cases}
\]
\[
 \int_{|w_i|<\hR}\psi
\left( \frac{\alpha \eta_i }{rt^{Q}}w_i^2  \right) dw_i=\hR\cG\left(\frac{\alpha \eta_i t^{2R}}{rt^Q}\right).
\]
Therefore,
\begin{equation}\label{oscformula2}
\begin{split}
I_{g,r}(\alpha,\vec{c})=\hR^4\prod_{i=1}^4 \left(
  \delta_{|\alpha|\leq \kappa_i} \delta_{\frac{|c_i|}{|g|} < \frac{|r|}{\hat{R}} }  +\delta_{\kappa_i<|\alpha|<1}  \psi (-\frac{t^Q{c_i}^2}{4rg^2\eta_i \alpha})\cG\left(\frac{\alpha \eta_i t^{2R}}{rt^Q}\right)\right).
\end{split}
\end{equation}
By our assumption, we have  $ \frac{|r|}{\hR} \leq  \kappa<  \hQ/\hR.$ Note that $Q=R+1$ and $q\kappa\geq \max \kappa_i\geq \kappa.$
This implies that $ I_{g,r}(\alpha,\vec{c})=0$ for every $|\alpha|\leq   \kappa .$  For $1>|\alpha|>  \kappa \geq  \frac{|r|}{\hR} ,$ we have  
\[
I_{g,r}(\alpha,\vec{c})=\hR^4\prod_{i=1}^4  \delta_{\kappa_i<|\alpha|<1}  \psi (-\frac{t^Q{c_i}^2}{4rg^2\eta_i \alpha})\cG\left(\frac{\alpha \eta_i t^{2R}}{rt^Q}\right).
\]
Hence, 
 \begin{equation*}
\begin{split}
I_{g,r}(\vec{c})&=\frac{\hQ\hR^{4}}{|r|}\sum_{\kappa\leq  \hl<1} \int _{ |\alpha|=\hl  }  \psi(\frac{-\alpha k}{rgt^Q})  \prod_{i=1}^4  \delta_{\kappa_i<|\alpha|<1}  \psi (-\frac{{t^Qc_i}^2}{4rg^2\eta_i \alpha}) 
\cG\left(\frac{\alpha \eta_i t^{2R}}{rt^Q}\right)
d\alpha.
\end{split}
\end{equation*}
By~\eqref{epsfact} and considering the sign of the Gauss sums,  we have for the Morgenstern quadratic form above the equality
\[
 \prod_{i}\cG\left(\frac{\alpha \eta_i t^{2R}}{rt^Q}\right) 
  =- \prod_{i}\min\left(1,\left(\frac{\hl \hR^2|\eta_i|}{|r|\hQ}\right)^{-1/2}\right),
\]
a quantity dependent not on $\alpha$ itself, but on the norm of $\alpha$ which is $\hat{l}$.
Hence,
\[
I_{g,r}(\vec{c})=-\frac{\hQ\hR^{4}}{|r|}\sum_{\kappa\leq \hat{l}<1}\prod_{i} \delta_{\kappa_i<\hl}  \min\left(1,\left(\frac{\hl \hR^2|\eta_i|}{|r|\hQ}\right)^{-1/2}\right) \int _{ |\alpha| =\hat{l}  }  \psi(\frac{-\alpha k}{rgt^Q})   \psi (-\frac{t^QF^*(\vec{c})}{4rg^2 \alpha})  d\alpha,
\]

where $F^*(\vec{c})=\sum_{i} \frac{{c_i}^2}{\eta_i}.$ Since we assume either 
\begin{enumerate}
\item$ \frac{|r|}{\hR} <  \kappa,$ 
\item $ \kappa= \frac{|r|}{\hR},  \text{ and } \max{(|c_3|,|c_4|)}\leq \max{(|c_1|,|c_2|)}\textit{ or } \deg(f) \text{ is odd}, $
\end{enumerate}
we have 
\[
\deg\left( \frac{kF^*(\vec{c})}{4r^2g^3}\right)\geq -2. 
\]
Let $a,b\in \mathbb{Z}$ where
\[
a=l+\deg(\frac{ k}{rgt^Q}), \text{ and } 2a+b=\deg\left( \frac{kF^*(\vec{c})}{4r^2g^3}\right). 
\]
Since $ 2a+b\geq -2$, either $b=0$ or $\max(a,a+b)>-1.$
 By Lemma~\ref{Bessel}, we have 
\begin{equation*}
\begin{split}
\int _{ |\alpha| =\hat{l}  }  \psi(\frac{-\alpha k}{rgt^Q})   \psi (-\frac{t^QF^*(\vec{c})}{4rg^2 \alpha}) d\alpha&= |\frac{rgt^Q}{ k}|B_{\infty}\left(\psi,l+\deg(\frac{ k}{rgt^Q}),\frac{kF^*(\vec{c})}{4r^2g^3}\right) 
\\
&=\begin{cases} 
 |\frac{rgt^Q}{ k}|\Kl_{\infty}\left(\psi, \frac{kF^*(\vec{c})}{4r^2g^3}\right) &\text{  if } 2l=\deg(\frac{t^{2Q}F^*(\vec{c})}{kg}),
 \\
 0 &\text{ otherwise.}
 \end{cases} 
\end{split}
\end{equation*}
Note that if $ 2l=\deg(\frac{t^{2Q}F^*(\vec{c})}{kg}),$ then $\hl> \kappa_i,$ and we have
\[
I_{g,r}(\vec{c})=- \hQ^2|g|^{2} |r|^{2}|F^*(\vec{c})|^{-1}\Kl_{\infty}\left(\psi, \frac{kF^*(\vec{c})}{4r^2g^3}\right).\qedhere
\]
\end{proof}

\section{The exponential sums $S_{g,r}(\vec{c})$}\label{section:phases}
In this section, we explicitly compute our exponential sums. Though our computations can be generalized to all non-degenerate quadratic forms, we focus here on the special case of the Morgenstern quadratic form
\[F(x_1,x_2,x_3,x_4):=\eta_1x_1^2+\eta_2x_2^2+\eta_3x_3^2+\eta_4x_4^2\]
over $\mathbb{F}_q[t]$, where $\eta_1=1,$ $\eta_2=-\nu,$ $\eta_3=-(t-1),$  $\eta_4=\nu(t-1)$ and  $\nu\in\mathbb{F}_q$ is not a square. Its dual quadratic form $F^*$ is obtained by inverting the coefficients $\eta_j$. Focusing on the Morgenstern quadratic form is no restriction in our case since we are primarily interested in proving upper bounds for the diameters of Morgenstern Ramanujan graphs. Throughout this section, we let $A=\diag(\eta_1,\hdots,\eta_4)$ be the diagonal matrix associated to this quadratic form. Also, by Lemma 3.2 of ~\cite{SZ2019}, $S_{g,r}(\vec{c})=0$ except possibly when $\vec{c}\equiv 2\beta(\vec{c})A\boldsymbol{\lambda}\bmod g$ for some $\beta(\vec{c})\in\mathcal{O}$.
\begin{proposition}\label{prop:phases} For the Morgenstern quadratic form above and $g\in\mathbb{F}_q[t]$ not divisible by $t-1$, we have that when $\gcd(r,t-1)|c_3,c_4$, then $S_{g,r}(\vec{c})$ is equal to
\begin{eqnarray*}&&\frac{|g|^4}{|m|^2}(|\gcd(r,t-1)|\tau_r\tau_{r/(r,t-1)})^2\psi\left(\frac{-\overline{mr}\beta(\vec{c})\frac{(f-F(\boldsymbol{\lambda}))}{m}-\overline{m^2r}\left<\boldsymbol{\lambda},\vec{c}\right>}{(g/m)^2}\right)\psi\left(\frac{\left<\boldsymbol{\lambda},\vec{c}\right>}{g^2r}\right)\\&&\cdot\sum_{s\in\mathcal{O}/(m)}\psi\left(\frac{-s\overline{g/m}\beta(\vec{c})}{m}\right)\Kl_{m^2r}\left(\overline{g/m}f-mrs,\frac{1}{4}\overline{g/m}^3F^*(\vec{c})\right),
\end{eqnarray*}
where $m:=(g,r^{\infty})$. If $\gcd(r,t-1)\nmid c_3,c_4$, then $S_{g,r}(\vec{c})=0$. Note that when $\vec{c}$ is such that $|\vec{c}|<|gr|$ (when $r\neq 1$ and $\kappa<\hQ/\hR=q$, for example), then $\psi\left(\frac{\left<\boldsymbol{\lambda},\vec{c}\right>}{g^2r}\right)=1$. 
\end{proposition}
\begin{proof}
Recall that, the exponential sums for quadratic forms in $d=4$ variables are
\begin{eqnarray*}S_{g,r}(\vec{c})&:=&\sum_{\substack{\ell\in\mathcal{O}\\|\ell|<|g|}}\sumstar_{|a|<|r|}\sum_{\vec{b}\in(\mathcal{O}/(gr))^4}\psi\left(\frac{(a+r\ell)(2\boldsymbol{\lambda}^TA\vec{b}-k)+agF(\vec{b})-\left<\vec{c},\vec{b}\right>}{gr}\right)\\ &=& \sum_{\substack{\ell\in\mathcal{O}\\|\ell|<|g|}}\sumstar_{|a|<|r|}\sum_{\vec{b}\in(\mathcal{O}/(gr))^4}\psi\left(\frac{(a+r\ell)(2\boldsymbol{\lambda}^TA\vec{b}-k+gF(\vec{b}))-\left<\vec{c},\vec{b}\right>}{gr}\right)
.\end{eqnarray*}
Summation over $\ell$ is zero except possibly when $g|2\boldsymbol{\lambda}^TA\vec{b}-k$, and so we may rewrite the latter as
\begin{eqnarray*}S_{g,r}(\vec{c})=\sum_{\substack{|a|<|gr|\\(a,r)=1}}\sum_{\substack{\vec{b}\in(\mathcal{O}/(gr))^4\\ g|2\boldsymbol{\lambda}^TA\vec{b}-k}}\psi\left(\frac{a(2\boldsymbol{\lambda}^TA\vec{b}-k+gF(\vec{b}))-\left<\vec{c},\vec{b}\right>}{gr}\right).\end{eqnarray*}
Let $m:=(g,r^{\infty})$. Using this expression for $S_{g,r}(\vec{c})$, we may rewrite the exponential sum in terms of $m$ as follows. First note that $gr=\frac{g}{m}\cdot(mr)$. Additionally, $(a,r)=1$ is equivalent to $(a,mr)=1$. Also, since $(m,\frac{g}{m})=1$, the condition $g|2\boldsymbol{\lambda}^TA\vec{b}-k$ is equivalent to the pair of conditions $\frac{g}{m}|2\boldsymbol{\lambda}^TA\vec{b}-k$ and $m|2\boldsymbol{\lambda}^TA\vec{b}-k$. Therefore, we have
\begin{eqnarray*}S_{g,r}(\vec{c})&=&\sum_{\substack{|a|<|gr|\\(a,r)=1}}\sum_{\substack{\vec{b}\in(\mathcal{O}/(gr))^4\\ g|2\boldsymbol{\lambda}^TA\vec{b}-k}}\psi\left(\frac{a(2\boldsymbol{\lambda}^TA\vec{b}-k+gF(\vec{b}))-\left<\vec{c},\vec{b}\right>}{gr}\right)\\&=& \frac{1}{|m|}\sum_{\substack{|a|<\left|\frac{g}{m}\cdot(mr)\right|\\(a,mr)=1}}\sum_{\substack{\vec{b}\in(\mathcal{O}/\left(\frac{g}{m}\cdot(mr)\right))^4\\ \frac{g}{m}|2\boldsymbol{\lambda}^TA\vec{b}-k}}\sum_{|s|<|m|}\psi\left(\frac{a(2\boldsymbol{\lambda}^TA\vec{b}-k+gF(\vec{b}))-\left<\vec{c},\vec{b}\right>}{gr}\right)\psi\left(\frac{s(2\boldsymbol{\lambda}^TA\vec{b}-k)}{m}\right)\\  &=& \frac{1}{|m|}\sum_{|s|<|m|}\sum_{\substack{|a|<\left|\frac{g}{m}\cdot(mr)\right|\\(a,mr)=1}}\sum_{\substack{\vec{b}\in(\mathcal{O}/\left(\frac{g}{m}\cdot(mr)\right))^4\\ \frac{g}{m}|2\boldsymbol{\lambda}^TA\vec{b}-k}}\psi\left(\frac{\left(a+\frac{grs}{m}\right)(2\boldsymbol{\lambda}^TA\vec{b}-k+gF(\vec{b}))-\left<\vec{c},\vec{b}\right>}{gr}\right)\\&=& \sum_{\substack{|a|<\left|\frac{g}{m}\cdot(mr)\right|\\(a,mr)=1}}\sum_{\substack{\vec{b}\in(\mathcal{O}/\left(\frac{g}{m}\cdot(mr)\right))^4\\ \frac{g}{m}|2\boldsymbol{\lambda}^TA\vec{b}-k}}\psi\left(\frac{a(2\boldsymbol{\lambda}^TA\vec{b}-k+gF(\vec{b}))-\left<\vec{c},\vec{b}\right>}{\frac{g}{m}\cdot (mr)}\right).\end{eqnarray*}
Since $\frac{g}{m}$ and $mr$ are coprime, we may write
\[a=(mr)a_1+\frac{g}{m}a_2\]
and
\[\vec{b}=(mr)\vec{b}_1+\frac{g}{m}\vec{b}_2,\]
where $a_1$ ranges modulo $\frac{g}{m}$, $a_2$ modulo $mr$ coprime to $mr$, $\vec{b}_1$ modulo $\frac{g}{m}$, $\vec{b}_2$ modulo $mr$. Furthermore, we also have the condition that $\frac{g}{m}|2\boldsymbol{\lambda}^TA\vec{b}-k$ which is equivalent to the condition that $\frac{g}{m}|2mr\boldsymbol{\lambda}^TA\vec{b}_1-k$. Then
\begin{eqnarray*}&&\psi\left(\frac{a(2\boldsymbol{\lambda}^TA\vec{b}-k+gF(\vec{b}))-\left<\vec{c},\vec{b}\right>}{\frac{g}{m}\cdot (mr)}\right)\\&=&\psi\left(\frac{a_1(2\boldsymbol{\lambda}^TA(mr\vec{b}_1)-k)-\left<\overline{mr}\vec{c},mr\vec{b}_1\right>}{\frac{g}{m}}\right)\psi\left(\frac{a_2(2\boldsymbol{\lambda}^TA(\frac{g}{m}\vec{b}_2)-k+gF(\frac{g}{m}\vec{b}_2))-\left<\overline{g/m}\vec{c},(g/m)\vec{b}_2\right>}{mr}\right).
\end{eqnarray*}
Since $\gcd(mr,\frac{g}{m})=1$, when summing over $\vec{b}_1$ and $\vec{b}_2$ we may replace $mr\vec{b}_1$ and $\frac{g}{m}\vec{b}_2$ with $\vec{b}_1$ and $\vec{b}_2$, respectively. Consequently, we have
\begin{eqnarray*}S_{g,r}(\vec{c})&=&\sum_{|a_1|<\left|\frac{g}{m}\right|}\sum_{\substack{\vec{b}_1\in(\mathcal{O}/(g/m))^4\\ \frac{g}{m}|2\boldsymbol{\lambda}^TA\vec{b}_1-k}}\psi\left(\frac{a_1(2\boldsymbol{\lambda}^TA(\vec{b}_1)-k)-\left<\overline{mr}\vec{c},\vec{b}_1\right>}{\frac{g}{m}}\right)\\&&\cdot\sumstar_{|a_2|<|mr|}\sum_{\vec{b}_2\in(\mathcal{O}/(mr))^4}\psi\left(\frac{a_2(2\boldsymbol{\lambda}^TA(\vec{b}_2)-k+gF(\vec{b}_2))-\left<\overline{g/m}\vec{c},\vec{b}_2\right>}{mr}\right).
\end{eqnarray*}
It is easy to see that the first summation is equal to
\[\left|\frac{g}{m}\right|^4\psi\left(\frac{-\overline{mr}\beta(\vec{c})k}{g/m}\right).\]
As a result,
\begin{eqnarray*}S_{g,r}(\vec{c})&=&\left|\frac{g}{m}\right|^4\psi\left(\frac{-\overline{mr}\beta(\vec{c})k}{g/m}\right)\sumstar_{|a_2|<|mr|}\sum_{\vec{b}_2\in(\mathcal{O}/(mr))^4}\psi\left(\frac{a_2(2\boldsymbol{\lambda}^TA(\vec{b}_2)-k+gF(\vec{b}_2))-\left<\overline{g/m}\vec{c},\vec{b}_2\right>}{mr}\right)\\ &=&\left|\frac{g}{m}\right|^4\psi\left(\frac{-\overline{mr}\beta(\vec{c})k}{g/m}\right)\sumstar_{|a|<|mr|} \psi\left(\frac{-ak}{mr}\right)  \prod_{j=1}^{4} \sum_{b\in\mathcal{O}/(mr)} \psi\left(\frac{ga\eta_jb^2+(2a\eta_j\lambda_j-\overline{g/m}c_j)b}{mr} \right).
\end{eqnarray*}
Let us denote the double summation in this expression by $S$ (temporarily neglect the leading factor). In order to complete our computation of a closed form for $S$, we use the following easy lemma whose proof we leave to the reader.
\begin{lemma}\label{lem:conggauss}Suppose $a,b,c\in\mathbb{F}_q[t]$. If $\gcd(a,c)\nmid b$, then
\[\sum_{x\in\mathcal{O}/(c)}\psi\left(\frac{ax^2+bx}{c}\right)=0.\]
\end{lemma}
Note that $\gcd(mr,ga\eta_j)=m\gcd(r,\eta_j)$. Using Lemma~\ref{lem:conggauss}, we should have $m\gcd(r,\eta_j)|2a\eta_j\lambda_j-\overline{g/m}c_j$ for every $j$ in order to have a nonzero value. Since $\Delta$ and $g$ are relatively prime, $m$ and $\gcd(r,\eta_j)$ are also relatively prime. Consequently, the condition that $m\gcd(r,\eta_j)|2a\eta_j\lambda_j-\overline{g/m}c_j$ for each $j$ can be rewritten as the pair of conditions $m|2aA\boldsymbol{\lambda}-\overline{g/m}\vec{c}$ and $\gcd(r,\boldsymbol{\eta})|\vec{c}$ (that is, for every $j$, $\gcd(r,\eta_j)|c_j$); otherwise, we have a zero value for $S_{g,r}(\vec{c})$. Let us assume for the rest of this section that $\gcd(r,\boldsymbol{\eta})|\vec{c}$. In this case,
\begin{eqnarray*}
S=|m|^4\sumstar_{\substack{|a|<|mr|\\ m|2aA\boldsymbol{\lambda}-\overline{g/m}\vec{c}}} \psi\left(\frac{-ak}{mr}\right)\prod_{j=1}^{4}\sum_{b\in\mathcal{O}/(r)} \psi\left(\frac{\frac{g}{m}a\frac{\eta_j}{(r,\eta_j)}b^2+\frac{(2a\frac{\eta_j}{(r,\eta_j)}\lambda_j-\overline{g/m}\frac{c_j}{(r,\eta_j)})}{m}b}{r/(r,\eta_j)} \right).
\end{eqnarray*}
Completing the square and using the computation of Gauss sums, we may rewrite
\begin{eqnarray*}
&&\sum_{b\in\mathcal{O}/(r)} \psi\left(\frac{\frac{g}{m}a\frac{\eta_j}{(r,\eta_j)}b^2+\frac{(2a\frac{\eta_j}{(r,\eta_j)}\lambda_j-\overline{g/m}\frac{c_j}{(r,\eta_j)})}{m}b}{r/(r,\eta_j)} \right)\\&=& |\gcd(r,\eta_j)|\left(\frac{(g/m)a(\eta_j/(r,\eta_j))}{r/(r,\eta_j)}\right)\tau_{r/(r,\eta_j)}\psi\left(\frac{-\overline{(g/m)a(\eta_j/(r,\eta_j))}\left(\frac{2a\frac{\eta_j}{(r,\eta_j)}\lambda_j-\overline{g/m}\frac{c_j}{(r,\eta_j)}}{2m}\right)^2}{r/(r,\eta_j)} \right),
\end{eqnarray*}
where $\left(\frac{(g/m)a(\eta_j/(r,\eta_j))}{r/(r,\eta_j)}\right)$ is the Jacobi symbol. Consequently, $S$ is equal to
\begin{eqnarray*}&&\left(\prod_{j=1}^4|\gcd(r,\eta_j)|\tau_{r/(r,\eta_j)}\right)|m|^{4}\\&&\cdot\sumstar_{\substack{|a|<|mr|\\ m|2aA\boldsymbol{\lambda}-\overline{g/m}\vec{c}}}\left(\prod_{j=1}^4\left(\frac{(g/m)a(\eta_j/(r,\eta_j))}{r/(r,\eta_j)}\right)\right)\psi\left(\frac{-ak}{mr}\right)\psi\left(\frac{-\sum_j\overline{(g/m)a(\eta_j/(r,\eta_j))}\left(\frac{2a\frac{\eta_j}{(r,\eta_j)}\lambda_j-\overline{g/m}\frac{c_j}{(r,\eta_j)}}{2m}\right)^2}{r/(r,\eta_j)} \right).\end{eqnarray*}
So far, all our computations were valid for general quadratic forms. In the rest of this proof, we restrict to the Morgenstern quadratic form. In this case, we can write $S$ more explicitly:
\[|m|^4\left(|\gcd(r,t-1)|\tau_r\tau_{r/(r,t-1)}\right)^2\psi\left(\frac{\overline{g/m}^2\left<\boldsymbol{\lambda},\vec{c}\right>}{m^2r}\right)\sumstar_{\substack{|a|<|mr|\\ m|2aA\boldsymbol{\lambda}-\overline{g/m}\vec{c}}}\psi\left(\frac{-\overline{g/m}fa-\frac{1}{4}\overline{g/m}^3F^*(\vec{c})\overline{a}}{m^2r} \right),\]
where the last equality follows from $km+\overline{g/m}F(\boldsymbol{\lambda})\equiv \overline{g/m}f\bmod m^2r$. Here, $F^*$ is the dual of $F$ as before in the computation of the oscillatory integrals. Furthermore, the condition $m|2aA\boldsymbol{\lambda}-\overline{g/m}\vec{c}$ is equivalent to $a\equiv \overline{g/m}\beta(\vec{c})\bmod m$. We deduce that
\begin{eqnarray*}S_{g,r}(\vec{c})&=&\frac{|g|^4}{|m|}\left(|\gcd(r,t-1)|\tau_r\tau_{r/(r,t-1)}\right)^2\psi\left(\frac{-\overline{mr}\beta(\vec{c})k}{g/m}\right)\psi\left(\frac{\overline{g/m}^2\left<\boldsymbol{\lambda},\vec{c}\right>}{m^2r}\right)\\&\cdot&\sumstar_{\substack{|a|<|m^2r|\\ a\equiv\overline{g/m}\beta(\vec{c})\bmod m}}\psi\left(\frac{-\overline{g/m}fa-\frac{1}{4}\overline{g/m}^3F^*(\vec{c})\overline{a}}{m^2r} \right),
\end{eqnarray*}
where we have changed summation over $a$ modulo $mr$ to modulo $m^2r$ at the cost of introducing a factor of $\frac{1}{|m|}$. We may replace the congruence condition by a summation modulo $m$ and rewrite $S_{g,r}(\vec{c})$ in terms of Kloosterman sums:
\begin{eqnarray*}S_{g,r}(\vec{c})&=&\frac{|g|^4}{|m|^2}\left(|\gcd(r,t-1)|\tau_r\tau_{r/(r,t-1)}\right)^2\psi\left(\frac{-\overline{mr}\beta(\vec{c})k}{g/m}\right)\psi\left(\frac{\overline{g/m}^2\left<\boldsymbol{\lambda},\vec{c}\right>}{m^2r}\right)\\&&\cdot\sum_{s\in\mathcal{O}/(m)}\psi\left(\frac{-s\overline{g/m}\beta(\vec{c})}{m}\right)\Kl_{m^2r}\left(\overline{g/m}f-mrs,\frac{1}{4}\overline{g/m}^3F^*(\vec{c})\right).
\end{eqnarray*}

By the Chinese Remainder Theorem,
\[\psi\left(\frac{-\overline{mr}\beta(\vec{c})k}{g/m}\right)\psi\left(\frac{\overline{g/m}^2\left<\boldsymbol{\lambda},\vec{c}\right>}{m^2r}\right)=\psi\left(\frac{-\overline{mr}\beta(\vec{c})\frac{(f-F(\boldsymbol{\lambda}))}{m}-\overline{m^2r}\left<\boldsymbol{\lambda},\vec{c}\right>}{(g/m)^2}\right)\psi\left(\frac{\left<\boldsymbol{\lambda},\vec{c}\right>}{g^2r}\right),\]
from which the conclusion follows.
\end{proof}

\section{Strong approximation and Ramanujan graphs}\label{morgenstern}
In this section, we begin by showing how a certain square-root cancellation in an exponential sum gives us strong approximation for non-degenerate quadratic forms in four variables over $\mathbb{F}_q[t]$. We then proceed to show that assuming the twisted Linnik\textendash Selberg Conjecture~\ref{tslconj}, we do have the desired square-root cancellation for Morgenstern's quadratic forms used in the construction of Ramanujan graphs with even degree (odd $q$).\\
\\
First, let us proceed to estimate the main term contributing to the smooth smooth $N(w,\boldsymbol{\lambda})$. The following lemmas are true for quadratic forms in more variables, but we restrict here to $d=4$.
\begin{lemma}\label{lem:Izerobound}Suppose $\varepsilon>0$. With the notation as before with $F$ a non-degenerate quadratic form in $4$ variables over $\mathbb{F}_q[t]$, and for $1\leq |r|\leq\hQ^{1-\varepsilon}$, we have
\[I_{g,r}(\vec{0})= C_F\hQ^4\]
for some non-negative constant $C_F$ and for sufficiently large (depending only on $\varepsilon$ and $F$) $\hQ$. $C_F>0$ if the system under consideration is solvable over $K_{\infty}$.
\end{lemma}
\begin{proof}
It follows from equation~\ref{intt} that
\[I_{g,r}(\vec{0})=\frac{\hQ}{|r|} \int_{\substack{|\vec{t}|< \hR\\ |G(\vec{t})|< \hQ|r|}} d\vec{t}=\frac{\hQ}{|r|} \int_{\substack{|g\vec{t}+\boldsymbol{\lambda}| \leq |f|^{1/2}\\ |F(g\vec{t}+\boldsymbol{\lambda})-f|<\hQ|r||g|^2}} d\vec{t}.\]
Making the substitution $\vec{x}=g\vec{t}+\boldsymbol{\lambda}$ gives us the equality
\[I_{g,r}(\vec{0})=\frac{\hQ}{|r||g|^4}\int_{|\vec{x}|\leq |f|^{1/2}:|F(\vec{x})-f|< \hQ|r||g|^2}d\vec{x}.\]
Write $f=\alpha_f u^2$, where $\alpha_f\in\{1,\nu,t,\nu t\}$ is the quadratic residue of $f$. By Lemma~\ref{lem:orthog} and Fubini, we may rewrite this as
\begin{eqnarray*}I_{g,r}(\vec{0})&=&\frac{\hQ}{|r||g|^4}\int_{|\vec{x}|\leq |f|^{1/2}}\int_{\TT}\psi\left(\frac{(F(\vec{x})-f)}{rg^2t^{Q}}\alpha\right)d\alpha d\vec{x}\\ &=&\frac{\hQ}{|r||g|^4}\int_{\TT}\int_{|\vec{x}|< \hat{D}}\psi\left(\frac{(F(\vec{x})-f)}{rg^2t^Q}\alpha\right)d\vec{x}d\alpha\\ &=& \frac{\hQ\hat{D}^4}{|r||g|^4}\int_{\TT}\int_{\TT^4}\psi\left(\frac{(F(\vec{x})-f/(t^2u^2))}{rg^2t^{Q}/(t^2u^2)}\alpha\right)d\vec{x}d\alpha\\ &=& \frac{\hQ\hat{D}^4}{|r||g|^4}\int_{\TT}\int_{\TT^4}\psi\left(\frac{(F(\vec{x})-\alpha_f/(t^2))}{rg^2t^{Q}/(t^2u^2)}\alpha\right)d\vec{x}d\alpha
\end{eqnarray*}
where $D:=1+\deg u$, and the last equality follows from scaling the $\vec{x}$ coordinate by a factor of $\hat{D}$. Making the substitution $\beta=\frac{\alpha}{rg^2t^{Q}/(t^2u^2)}$, we obtain the equality
\[I_{g,r}(\vec{0})=\frac{\hQ^2\hat{D}^4}{|g|^{2}\hat{2D}}\int_{|\beta|<\frac{\hat{2D}}{\hQ|r||g|^2}}\int_{\TT^4}\psi\left((F(\vec{x})-\alpha_f/t^2)\beta\right)d\vec{x}d\beta.\]
Note that the integral is equal to
\begin{eqnarray*}\frac{\hat{2D}}{\hQ|r||g|^2}\vol\left(\left\{\vec{x}\in\TT^4:|F(\vec{x})-\alpha_f/t^{2}|\leq \frac{\hQ|r||g|^2}{\hat{2D}}\right\}\right)\geq 0.
\end{eqnarray*}
Note that $\frac{\hat{2D}}{\hQ|r||g|^2}\gg \hQ^{\varepsilon}$. Applying Lemma 6.2 of~\cite{SZ2019}, we can choose $\hQ$ large enough (depending on $\varepsilon$ and the $F$) so that the integral over $|\beta|$ is constant over balls of radii at least $\hQ^{\varepsilon}$. The conclusion follows.
\end{proof}
As in Lemma 6.3 of~\cite{SZ2019}, we can show that for $\hQ^{1-\varepsilon}\leq |r|\leq\hQ$, the contribution of the terms in $N(w,\boldsymbol{\lambda})$ when $\vec{c}=\vec{0}$ and such $r$ is small.
\begin{lemma}\label{lem:small} 
\[\sum_{\hQ^{1-\varepsilon}\leq |r|\leq\hQ}|gr|^{-4}|S_{g,r}(\vec{0})||I_{g,r}(\vec{0})|\ll_{\varepsilon,F}|g|^{\varepsilon}\hQ^{\frac{7}{2}+\varepsilon}\]
\end{lemma}
\begin{proof}
The only difference in the proof of this lemma and that of Lemma 6.3 of~\cite{SZ2019} is that the definitions of the oscillatory integrals are different. However, we only need the same bound
\[|I_{g,r}(\vec{0})|\ll_{\varepsilon,F} \hQ^{4+\varepsilon}\]
for such $r$, which is trivial. The rest of the proof is as before; we also need to use Proposition 3.1 of~\cite{SZ2019}.
\end{proof}
We now proceed to show that
\[\sum_{r:1\leq |r|\leq \hat{T}}|gr|^{-4}S_{g,r}(\vec{0})\]
can be written in terms of local densities. Indeed, by Lemma 6.5 of~\cite{SZ2019}, we have the estimate
\[\sum_{r:1\leq |r|\leq \hat{T}}|gr|^{-4}S_{g,r}(\vec{0})=\sum_{r}|gr|^{-4}S_{g,r}(\vec{0})+O_{\varepsilon,\Delta}(\hat{T}^{-1/2+\varepsilon})\]
for every $\varepsilon>0$. On the other hand, Lemma 6.6 of \textit{loc.cit} shows that
\[\sum_{r}|gr|^{-4}S_{g,r}(\vec{0})=\prod_{\varpi}\sigma_{\varpi} \gg |f|^{-\varepsilon},\]
where $\varpi$ ranges over the monic irreducible polynomials in $\mathbb{F}_q[t]$, and
\[\sigma_{\varpi}:=\lim_{k\rightarrow\infty}\frac{|\left\{\vec{x}\bmod \varpi^{k+\nu_\varpi(g)}:F(\vec{x})\equiv f\bmod\varpi^{k+\nu_\varpi(g)},\ \vec{x}\equiv\boldsymbol{\lambda}\bmod\varpi^{\nu_{\varpi}(g)}\right\}|}{|\varpi|^{3k}}.\]
\\
Using the above, estimates, let us take the simple step of showing that conditional on a square-root cancellation we have optimal strong approximation for any non-degenerate quadratic form in four variables over $\mathbb{F}_q[t]$. First, recall from Section~\ref{deltamethod} that the smooth sum for quadratic forms of four variables is
\[N(w,\boldsymbol{\lambda})=\frac{1}{|g|\hQ^2}\sum_{\substack{
r\in \cO\\
|r|\leq \hat Q
} }\sum_{\vec{c}\in\mathcal{O}^4}|gr|^{-4}S_{g,r}(\vec{c})I_{g,r}(\vec{c}).\]
\begin{lemma}Suppose we have a non-degenerate quadratic form $F$ over $\mathbb{F}_q[t]$ in $d=4$ variables. Additionally,\textup{ assume} that
\[\left|\sum_{1\leq |r|\leq\hQ}\sum_{\vec{c}\neq 0}|gr|^{-4}S_{g,r}(\vec{c})I_{g,r}(\vec{c})\right|\ll_{\varepsilon,F} \hQ^{\frac{7}{2}+\varepsilon}|g|^{\frac{1}{2}+\varepsilon}.\]
Then Conjecture~\ref{mainconj} is true for the given quadratic form $F$.
\end{lemma}
\begin{proof}
Using the assumption, Lemma~\ref{o}, and Lemma~\ref{lem:small}, we have
\[N(w,\boldsymbol{\lambda})=\frac{1}{|g|\hQ^2}\sum_{\substack{r\in \cO\\|r|\leq \hat Q^{1-\varepsilon}}}|gr|^{-4}S_{g,r}(\vec{0})I_{g,r}(\vec{0})+O_{\varepsilon,F}\left(\hQ^{\frac{3}{2}+\varepsilon}|g|^{-\frac{1}{2}+\varepsilon}\right).\]
From Lemma~\ref{lem:Izerobound} above, $I_{g,r}(\vec{0})=C_F\hQ^4$ for some constant $C_F>0$ and sufficiently large (depending only on $\varepsilon$ and $F$) $\hQ$. Hence for such $\hQ$,
\[\frac{1}{|g|\hQ^2}\sum_{\substack{r\in \cO\\|r|\leq \hat Q^{1-\varepsilon}}}|gr|^{-4}S_{g,r}(\vec{0})I_{g,r}(\vec{0})=\frac{C_F\hQ^{2}}{|g|}\sum_{\substack{r\in \cO\\|r|\leq \hat Q^{1-\varepsilon}}}|gr|^{-4}S_{g,r}(\vec{0}).\]
From the discussion prior to this lemma, we know that we can express the first sum in terms of local densities: 
\[\sum_{\substack{r\in \cO\\|r|\leq \hat Q^{1-\varepsilon}}}|gr|^{-4}S_{g,r}(\vec{0})=\prod_{\varpi}\sigma_{\varpi}+O\left(\hQ^{-\frac{1}{2}+\varepsilon}\right).\] 
Consequently, the smooth sum is equal to
\begin{eqnarray*}N(w,\boldsymbol{\lambda})&=&\frac{C_F\hQ^{2}}{|g|}\left(\prod_{\varpi}\sigma_{\varpi}+O\left(\hQ^{-\frac{1}{2}+\varepsilon}\right)\right)+O_{\varepsilon,F}\left(\hQ^{\frac{3}{2}+\varepsilon}|g|^{-\frac{1}{2}+\varepsilon}\right)\\&=&\frac{C_F\hQ^{2}}{|g|}\prod_{\varpi}\sigma_{\varpi}+O_{\varepsilon,F}\left(\hQ^{\frac{3}{2}+\varepsilon}|g|^{-\frac{1}{2}+\varepsilon}\right)\\&=& \frac{C_F\hQ^{2}}{|g|}\prod_{\varpi}\sigma_{\varpi}\left(1+O_{\varepsilon,F}\left(\frac{|f|^{\varepsilon}\hQ^{\frac{3}{2}+\varepsilon}|g|^{\frac{1}{2}+\varepsilon}}{\hQ^{2}}\right)\right)\\&=&\frac{C_F\hQ^{2}}{|g|}\prod_{\varpi}\sigma_{\varpi}\left(1+O_{\varepsilon,F}\left(\frac{|g|^{1+\varepsilon}}{|f|^{\frac{1}{4}-\varepsilon}}\right)\right)\\&=& \frac{C_F\hQ^{2}}{|g|}\prod_{\varpi}\sigma_{\varpi}\left(1+O_{\varepsilon,F}\left(\left(\frac{|g|^{4+\varepsilon}}{|f|}\right)^{\frac{1}{4}-\varepsilon}\right)\right).\end{eqnarray*}
Hence, if $|f|\gg |g|^{4+\varepsilon}$, we have strong approximation. In the third equality, we have used that the product of the local densities is $\gg |f|^{-\varepsilon}$.
\end{proof}
\begin{remark}
In the proof of the main theorem of ~\cite{SZ2019}, the only reason we had to have $|f|\gg |g|^{6+\varepsilon}$ in the case of four variables was that we used the weaker statement
\[\left|\sum_{1\leq |r|\leq\hQ}\sum_{\vec{c}\neq 0}|gr|^{-d}S_{g,r}(\vec{c})I_{g,r}(\vec{c})\right|\leq \sum_{1\leq |r|\leq\hQ}\sum_{\vec{c}\neq 0}|gr|^{-d}|S_{g,r}(\vec{c})||I_{g,r}(\vec{c})|\ll_{\varepsilon,F} \hQ^{\frac{d+3}{2}+\varepsilon}|g|^{\frac{d-3}{2}+\varepsilon}(1+|g|^{-\frac{d-5}{2}+\varepsilon})\]
proved in Proposition 7.1 of \textit{loc.cit}. Indeed, for $d=4$, this is weaker than what we ask above because then $1+|g|^{-\frac{d-5}{2}+\varepsilon}\sim |g|^{1/2+\varepsilon}$ and so is not $O(1)$ (in contrast to when $d\geq 5$). The bound we assume in the statement of this lemma is precisely that if we do not take absolute values we get an extra power saving of $|g|^{1/2}$ when in the case of four variables.
\end{remark}
\begin{remark}
In light of the previous remark and the proof of the above lemma, any improvement to the factor $|g|^{1/2}$ in the previous remark would allow us to weaken the condition $\deg f\geq (6+\varepsilon)\deg g+O_{\varepsilon}(1)$ that was required for the main theorem of \cite{SZ2019} in the case of non-degenerate quadratic forms in $d=4$ variables.
\end{remark}
In the rest of this section, we show how the twisted Linnik\textendash Selberg Conjecture~\ref{tslconj} implies that the above square-root cancellation is true for Morgenstern's quadratic forms. This in turn implies Conjecture~\ref{conj:graphs} giving the upper bound $\left(\frac{4}{3}+\varepsilon\right)\log_q|G|+O_{\varepsilon}(1)$ for the diameter of $q+1$-regular Morgenstern Ramanujan graphs with $q$ odd.\\
\\
Recall the notations in Section~\ref{osil}. In order to understand the error in the smooth sum $N(w,\boldsymbol{\lambda})$ for the Morgenstern quadratic form, we use the explicit formulas for the oscillatory integrals $I_{g,r}(\vec{c})$ and exponential sums derived in the last two sections. Note that by Lemma~\ref{o}, $I_{g,r}(\vec{c})=0$ when $|r|\leq\hQ$ and $|\vec{c}|\geq\hQ|g|/\hR$. Therefore, it suffices to study
\[\sum_{1\leq |r|\leq\hQ}\sum_{0<|\vec{c}|<\frac{\hQ|g|}{\hR}}|gr|^{-d}S_{g,r}(\vec{c})I_{g,r}(\vec{c})=E_1+E_2,\]
where
\[E_1:=\sum_{0<|\vec{c}|<\hQ|g|/\hR}\sum_{1\leq |r|\leq\frac{\hR |\vec{c}|q^{\pi_{\vec{c}}-1}}{|g|}}|gr|^{-4}S_{g,r}(\vec{c})I_{g,r}(\vec{c})\]
and
\[E_2:=\sum_{0<|\vec{c}|<\hQ|g|/\hR}\sum_{\substack{\frac{\hR |\vec{c}|q^{\pi_{\vec{c}}}}{|g|}\leq |r|\leq\hQ}}|gr|^{-4}S_{g,r}(\vec{c})I_{g,r}(\vec{c}).\]
Here, $\pi_{\vec{c}}=0$ if $\vec{c}$ satisfies $\max(|c_3|,|c_4|)>\max(|c_1|,|c_2|)$ and $\deg f$ is even; otherwise, $\pi_{\vec{c}}=1$. In order to obtain the desired bound in the above lemma, it suffices to prove the desired bound for each of $E_1$ and $E_2$ separately. For simplicity, we assume for the rest of this section that \textit{$g$ is an irreducible polynomial} in $\mathbb{F}_q[t]$.\\
\\
We first treat $E_2$. Note that by Proposition~\ref{prop:oscmorgen}, for $\frac{\hR|\vec{c}|q^{\pi_{\vec{c}}}}{|g|}\leq |r|\leq\hQ$, we have $I_{g,r}(\vec{c})=I_{g,r}(\vec{0})$ or $0$. Furthermore, from the definition of $I_{g,r}(\vec{0})$, we know that it depends on $|r|$, not $r$ itself. Therefore, it makes sense to write $I_{g,|r|}(\vec{0})$ instead of $I_{g,r}(\vec{0})$. From this discussion, we obtain
\[|E_2|\leq \sum_{0<|\vec{c}|<\hQ|g|/\hR}\sum_{\substack{\frac{\hR |\vec{c}|q^{\pi_{\vec{c}}}}{|g|}\leq \hat{T}\leq\hQ}}|I_{g,\hat{T}}(\vec{0})||g|^{-4}\hat{T}^{-4}\left|\sum_{|r|=\hat{T}}S_{g,r}(\vec{c})\right|.\]
In the following, we will use the following notation: $\sum^{\textup{exc}}$ denotes summation over those $\vec{c}$ such that $S_{g,r}(\vec{c})\neq 0$ and $|\vec{c}|\leq |g|$. We have the following accompanying lemma.
\begin{lemma}\label{boundonvecc} For every $\theta<0$ and every $0\leq T\leq \deg g$, we have
\[\sum^{\textup{exc}}_{0<|\vec{c}|\leq \hat{T}}|\vec{c}|^{\theta}\ll_{\varepsilon,F,\theta}\hat{T}^{\theta+1}.\]
\end{lemma}
\begin{proof}
By Lemma 3.2 of~\cite{SZ2019}, we know that these must be polynomial multiples of $A\boldsymbol{\lambda}$ modulo $g$. Since $0<|\vec{c}|\leq |g|$, $\vec{c}$ and $\alpha$ uniquely determine each other. By assumption, at least one coordinate of $\boldsymbol{\lambda}$ is relatively prime to $g$, from which the final inequality follows. 
\end{proof}
Suppose we have for every $T\geq 0$,
\begin{equation}\label{boundwewant}\left|\sum_{|r|=\hat{T}}S_{g,r}(\vec{c})\right|\ll_{\varepsilon,F} \hQ^{\varepsilon}|g|^{4+\varepsilon}\hat{T}^{3+\varepsilon}.
\end{equation}
Then it would follow from this assumption and Lemma~\ref{boundonvecc} that
\begin{eqnarray*}|E_2|&\leq& \sum_{0<|\vec{c}|<\hQ|g|/\hR}\sum_{\substack{\frac{\hR |\vec{c}|q^{\pi_{\vec{c}}}}{|g|}\leq \hat{T}\leq\hQ}}|I_{g,\hat{T}}(\vec{0})||g|^{-4}\hat{T}^{-4}\left|\sum_{|r|=\hat{T}}S_{g,r}(\vec{c})\right|\\ &\ll_{\varepsilon,F}& \hQ^{4+\varepsilon}|g|^{\varepsilon}\sum^{\textup{exc}}_{0<|\vec{c}|<\hQ|g|/\hR}\sum_{\substack{\frac{\hR |\vec{c}|q^{\pi_{\vec{c}}}}{|g|}\leq \hat{T}\leq\hQ}}\hat{T}^{-1+\varepsilon}\\ &\ll_{\varepsilon,F}& \hQ^{4+\varepsilon}|g|^{\varepsilon}\sum^{\textup{exc}}_{0<|\vec{c}|<\hQ|g|/\hR}\left(\frac{\hQ |\vec{c}|}{|g|}\right)^{-1+\varepsilon}\\ &\ll_{\varepsilon,F}& \hQ^{3+\varepsilon}|g|^{1+\varepsilon}\sum^{\textup{exc}}_{0<|\vec{c}|<\hQ|g|/\hR}|\vec{c}|^{-1+\varepsilon}\ll_{\varepsilon,F} \hQ^{7/2+\varepsilon}|g|^{1/2+\varepsilon},
\end{eqnarray*}
Note that $\hQ$ and $\hR$ differ by a factor of $q$. This is exactly the desired bound on $E_2$. Therefore, we have reduced to showing inequality~\eqref{boundwewant} for each integer $T\geq 0$.\\
\\
By Proposition~\ref{prop:phases}, when $\gcd(r,t-1)|c_3,c_4$, $S_{g,r}(\vec{c})$ is equal to
\begin{equation}\label{Sgr}
\begin{split}
&\frac{|g|^4}{|m|^2}(|\gcd(r,t-1)|\tau_r\tau_{r/(r,t-1)})^2\psi\left(\frac{-\overline{mr}\beta(\vec{c})\frac{(f-F(\boldsymbol{\lambda}))}{m}-\overline{m^2r}\left<\boldsymbol{\lambda},\vec{c}\right>}{(g/m)^2}\right)\psi\left(\frac{\left<\boldsymbol{\lambda},\vec{c}\right>}{g^2r}\right)\\&\cdot\sum_{s\in\mathcal{O}/(m)}\psi\left(\frac{-s\overline{g/m}\beta(\vec{c})}{m}\right)\Kl_{m^2r}\left(\overline{g/m}f-mrs,\frac{1}{4}\overline{g/m}^3F^*(\vec{c})\right),
\end{split}
\end{equation}
where $m:=(g,r^{\infty})$. If $\gcd(r,t-1)\nmid c_3,c_4$, then $S_{g,r}(\vec{c})=0$. Since $g$ is irreducible, $m=1$ or $m=g$. For the bound on $E_2$, we split the sum over $r$ such that $m=1$ and $m=g$, and show the desired bounds separately. First, let us show that we may assume that $m=1$. 
\begin{lemma}\label{g|r}For every $0\leq T\leq \deg g$,
\[\sum_{\substack{|r|=\hat{T}\\g|r}}|S_{g,r}(\vec{c})|\ll_{\varepsilon,F} \hQ^{\varepsilon}|g|^{3+\varepsilon}\hat{T}^{7/2+\varepsilon},\]
from which it follows that the contributions to $E_2$ from those $r$ such that $g|r$ may be neglected.
\end{lemma}
\begin{proof}
When $g|r$, $m=g$. From the expression~\eqref{Sgr} for $S_{g,r}(\vec{c})$, we see that it suffices to show that
\begin{eqnarray*}
&&\sum_{\substack{|r|=\hat{T}\\ g|r}}(|\gcd(r,t-1)|\tau_r\tau_{r/(r,t-1)})^2\left|\sum_{s\in\mathcal{O}/(g)}\psi\left(\frac{-s\beta(\vec{c})}{g}\right)\Kl_{g^2r}\left(f-grs,\frac{1}{4}F^*(\vec{c})\right)\right|\\&\ll_{\varepsilon,F}& \hQ^{\varepsilon}|g|^{1+\varepsilon}\hat{T}^{7/2+\varepsilon}.
\end{eqnarray*}
Note that in this case, $|\boldsymbol{\lambda}|<|g|$, $|\vec{c}|\leq |g|$, and $g|r$, and so $\psi\left(\frac{\left<\boldsymbol{\lambda},\vec{c}\right>}{g^2r}\right)=1$. By the Weil bound on Kloosterman sums (Lemma 3.5 of~\cite{SZ2019}), we have that
\begin{eqnarray*} 
\left|\sum_{s\in\mathcal{O}/(g)}\psi\left(\frac{-s\beta(\vec{c})}{g}\right)\Kl_{g^2r}\left(f-grs,\frac{1}{4}F^*(\vec{c})\right)\right|\ll_{\varepsilon,F} |g|^{2+\varepsilon}|r|^{1/2+\varepsilon}|\gcd(f-grs,\frac{1}{4}F^*(\vec{c}),g^2r)|^{1/2}.
\end{eqnarray*}
Since $\gcd(g,f)=1$ and $g|r$ by assumption, $\gcd(f-grs,\frac{1}{4}F^*(\vec{c}),g^2r)=\gcd(f,\frac{1}{4}F^*(\vec{c}),r/g)$. Therefore,
\begin{eqnarray*}
&&\sum_{\substack{|r|=\hat{T}\\ g|r}}(|\gcd(r,t-1)|\tau_r\tau_{r/(r,t-1)})^2\left|\sum_{s\in\mathcal{O}/(g)}\psi\left(\frac{-s\beta(\vec{c})}{g}\right)\Kl_{g^2r}\left(f-grs,\frac{1}{4}F^*(\vec{c})\right)\right|\\&\ll_{\varepsilon,F}& |g|^{2+\varepsilon}\hat{T}^{2+1/2+\varepsilon}\sum_{\substack{|r|=\hat{T}\\ g|r\\ \gcd(r,t-1)|c_3,c_4}}|\gcd(f,r/g)|^{1/2}\ll_{\varepsilon,F} |g|^{2+\varepsilon}\hat{T}^{2+1/2+\varepsilon}|f|^{\varepsilon}\frac{\hat{T}}{|g|}\ll_{\varepsilon,F}\hQ^{\varepsilon}|g|^{1+\varepsilon}\hat{T}^{7/2+\varepsilon}.
\end{eqnarray*}
We use this to show that the contribution to $E_2$ from those $r$ such that $g|r$ already satisfies the desired bound on $E_2$. Indeed, we have
\begin{eqnarray*}&&\sum_{0<|\vec{c}|<\hQ|g|/\hR}\sum_{\substack{\frac{\hR |\vec{c}|q^{\pi_{\vec{c}}}}{|g|}\leq \hat{T}\leq\hQ}}|I_{g,\hat{T}}(\vec{0})|\hat{T}^{-4}|g|^{-4}\sum_{\substack{|r|=\hat{T}\\ g|r}}|S_{g,r}(\vec{c})|\\ &\ll_{\varepsilon,F}& \hQ^{4+\varepsilon}|g|^{-1+\varepsilon}\sum^{\textup{exc}}_{0<|\vec{c}|<\hQ|g|/\hR}\sum_{\substack{\frac{\hR |\vec{c}|q^{\pi_{\vec{c}}}}{|g|}\leq \hat{T}\leq\hQ}}\hat{T}^{-1/2+\varepsilon}\\ &\ll_{\varepsilon,F}& \hQ^{4+\varepsilon}|g|^{-1+\varepsilon}\sum^{\textup{exc}}_{0<|\vec{c}|<\hQ|g|/\hR}\left(\frac{\hQ |\vec{c}|}{|g|}\right)^{-1/2+\varepsilon}\\ &\ll_{\varepsilon,F}& \hQ^{7/2+\varepsilon}|g|^{-1/2+\varepsilon}\sum^{\textup{exc}}_{0<|\vec{c}|<\hQ|g|/\hR}|\vec{c}|^{-1/2+\varepsilon}\ll_{\varepsilon,F} \hQ^{7/2+\varepsilon}|g|^{\varepsilon},
\end{eqnarray*}
where the final bound follows from Lemma~\ref{boundonvecc}.
\end{proof}
From Lemma~\ref{g|r}, we may assume that $m=1$, that is, $g\nmid r$. Using this and the expression~\eqref{Sgr} with $m=1$, for the inequality~\eqref{boundwewant} it suffices to show that
\begin{eqnarray*}&&\left|\sum_{\substack{|r|=\hat{T}\\ (g,r)=1\\ \gcd(r,t-1)|c_3,c_4}}(|\gcd(r,t-1)|\tau_r\tau_{r/(r,t-1)})^2\psi\left(\frac{-\overline{r}(\beta(\vec{c})(f-F(\boldsymbol{\lambda}))+\left<\boldsymbol{\lambda},\vec{c}\right>)}{g^2}\right)\Kl_{r}\left(\overline{g}f,\frac{1}{4}\overline{g}^3F^*(\vec{c})\right)\right|\\&\ll_{\varepsilon,F}& |g|^{\varepsilon}\hat{T}^{3+\varepsilon}.
\end{eqnarray*}
Note that since $|\boldsymbol{\lambda}|<|g|$, $|\vec{c}|\leq |g|$, $\psi\left(\frac{\left<\boldsymbol{\lambda},\vec{c}\right>}{g^2r}\right)=1$ unless possibly when $r=1$, which contributes a term of norm $1$ to the above sum. This is why we may suppress the $\psi\left(\frac{\left<\boldsymbol{\lambda},\vec{c}\right>}{g^2r}\right)$.\\
\\
We split the sum into two sums, one where $t-1|r$ and one where $t-1\nmid r$. Summing over those $r$ such that $t-1|r$ gives us the sum
\begin{eqnarray*}q^2\sum_{\substack{|r|=\hat{T}\\ (g,r)=1,\ t-1|r\\ t-1|c_3,c_4}}(\tau_r\tau_{r/(r,t-1)})^2\psi\left(\frac{-\overline{r}(\beta(\vec{c})(f-F(\boldsymbol{\lambda}))+\left<\boldsymbol{\lambda},\vec{c}\right>)}{g^2}\right)\Kl_{r}\left(\overline{g}f,\frac{1}{4}\overline{g}^3F^*(\vec{c})\right).
\end{eqnarray*}
Since $\tau_r^2\tau_{r/(t-1)}^2$ only depends on $|r|$, and not $r$ itself, we can pull it out of the sum. This term has norm $\hat{T}^2/q$. The second sum, that is when $t-1\nmid r$, is
\begin{eqnarray*}\hat{T}^2\sum_{\substack{|r|=\hat{T}\\ ((t-1)g,r)=1}}\psi\left(\frac{-\overline{r}(\beta(\vec{c})(f-F(\boldsymbol{\lambda}))+\left<\boldsymbol{\lambda},\vec{c}\right>)}{g^2}\right)\Kl_{r}\left(\overline{g}f,\frac{1}{4}\overline{g}^3F^*(\vec{c})\right).
\end{eqnarray*}
Therefore, in any case, it suffices to show that we have the following two cancellations. First, that if $t-1|c_3,c_4$,
\begin{eqnarray*}\left|\sum_{\substack{|r|=\hat{T}\\ (g,r)=1,\ t-1|r}}\psi\left(\frac{-\overline{r}(\beta(\vec{c})(f-F(\boldsymbol{\lambda}))+\left<\boldsymbol{\lambda},\vec{c}\right>)}{g^2}\right)\Kl_{r}\left(\overline{g}f,\frac{1}{4}\overline{g}^3F^*(\vec{c})\right)\right|\ll_{\varepsilon,F} \hQ^{\varepsilon}|g|^{\varepsilon}\hat{T}^{1+\varepsilon}.
\end{eqnarray*}
Second, that
\begin{eqnarray*}\left|\sum_{\substack{|r|=\hat{T}\\ ((t-1)g,r)=1}}\psi\left(\frac{-\overline{r}(\beta(\vec{c})(f-F(\boldsymbol{\lambda}))+\left<\boldsymbol{\lambda},\vec{c}\right>)}{g^2}\right)\Kl_{r}\left(\overline{g}f,\frac{1}{4}\overline{g}^3F^*(\vec{c})\right)\right|\ll_{\varepsilon,F} \hQ^{\varepsilon}|g|^{\varepsilon}\hat{T}^{1+\varepsilon}.
\end{eqnarray*}
Of course, we may replace one of the above bounds, say the second one, with
\begin{eqnarray*}\left|\sum_{\substack{|r|=\hat{T}\\ (g,r)=1}}\psi\left(\frac{-\overline{r}(\beta(\vec{c})(f-F(\boldsymbol{\lambda}))+\left<\boldsymbol{\lambda},\vec{c}\right>)}{g^2}\right)\Kl_{r}\left(\overline{g}f,\frac{1}{4}\overline{g}^3F^*(\vec{c})\right)\right|\ll_{\varepsilon,F} \hQ^{\varepsilon}|g|^{\varepsilon}\hat{T}^{1+\varepsilon}.
\end{eqnarray*}

We now use the explicit computation of $I_{g,r}(\vec{c})$ to show that a different kind of twisted Linnik\textendash Selberg cancellation over function fields, taking the infinite place into account as well, gives us the desired bound on $E_1$. Indeed, by Proposition~\ref{prop:oscmorgen}, for $\vec{c}$ and $r$ such that $1\leq |r|\leq \frac{\hR|\vec{c}|q^{\pi_{\vec{c}}-1}}{|g|}$, we have $I_{g,r}(\vec{c})=0$ or
\[I_{g,r}(\vec{c})=-\hQ^2|gr|^{2}|F^*(\vec{c})|^{-1}\Kl_{\infty}\left(\psi, \frac{fF^*(\vec{c})}{4r^2g^4}\right).\]
Therefore,
\[|E_1|\leq \hQ^2\sum_{0<|\vec{c}|<\hQ|g|/\hR}|F^*(\vec{c})|^{-1}\left|\sum_{1\leq |r|\leq\frac{\hR |\vec{c}|q^{\pi_{\vec{c}}-1}}{|g|}}|gr|^{-2}S_{g,r}(\vec{c})\Kl_{\infty}\left(\psi, \frac{fF^*(\vec{c})}{4r^2g^4}\right)\right|.\]
We can use Proposition~\ref{prop:phases} to rewrite this inequality as

\begin{eqnarray*}|E_1|&\leq&\hQ^2|g|^2\sum_{m|g}\sum_{0<|\vec{c}|<\hQ|g|/\hR}|F^*(\vec{c})|^{-1}\Bigg|\sum_{\substack{1\leq |r|\leq\frac{\hR |\vec{c}|q^{\pi_{\vec{c}}-1}}{|g|}\\ (g,r^{\infty})=m\\ \gcd(r,t-1)|c_3,c_4}}\frac{1}{|mr|^2}(|\gcd(r,t-1)|\tau_r\tau_{r/(r,t-1)})^2\\&&\cdot\psi\left(\frac{-\overline{mr}\beta(\vec{c})\frac{(f-F(\boldsymbol{\lambda}))}{m}-\overline{m^2r}\left<\boldsymbol{\lambda},\vec{c}\right>}{(g/m)^2}\right)\psi\left(\frac{\left<\boldsymbol{\lambda},\vec{c}\right>}{g^2r}\right)\sum_{s\in\mathcal{O}/(m)}\psi\left(\frac{-s\overline{g/m}\beta(\vec{c})}{m}\right)\\&&\cdot\Kl_{m^2r}\left(\overline{g/m}f-mrs,\frac{1}{4}\overline{g/m}^3F^*(\vec{c})\right)\Kl_{\infty}\left(\psi, \frac{fF^*(\vec{c})}{4r^2g^4}\right)\Bigg|.
\end{eqnarray*}

Using the fact that Morgenstern quadratic forms are anisotropic and so satisfy $|\vec{c}|^2\ll_{F} |F^*(\vec{c})|$, we may reduce, as in the case of $E_2$, to showing that
\begin{eqnarray*}&&\Bigg|\sum_{\substack{|r|=\hat{T}\\ (g,r)=1\\ \gcd(r,t-1)|c_3,c_4}}(|\gcd(r,t-1)|\tau_r\tau_{r/(r,t-1)})^2\psi\left(\frac{-\overline{r}(\beta(\vec{c})(f-F(\boldsymbol{\lambda}))+\left<\boldsymbol{\lambda},\vec{c}\right>)}{g^2}\right)\\&&\cdot\Kl_{r}\left(\overline{g}f,\frac{1}{4}\overline{g}^3F^*(\vec{c})\right)\Kl_{\infty}\left(\psi, \frac{fF^*(\vec{c})}{4r^2g^4}\right)\Bigg|\ll_{\varepsilon,F} \hQ^{\varepsilon}|g|^{\varepsilon}\hat{T}^{3+\varepsilon}.\end{eqnarray*}
When $m:=(g,r)=g$, we may argue as before and use the Weil bound $\Kl_{\infty}(\psi,\alpha)\ll_{\varepsilon} |\alpha|^{1/4+\varepsilon}$ (Lemma 5.8 of~\cite{SZ2019}) in addition to $|\vec{c}|\leq |g|$. Therefore, as in the case of $E_2$, we may assume that $m=1$, that is, $g\nmid r$.\\
\\
As in the case of $E_2$, we may split into two sums, one where $t-1|r$ and one where $t-1\nmid r$. We similarly obtain that it suffices to show that
\begin{eqnarray*}\left|\sum_{\substack{|r|=\hat{T}\\ (g,r)=1,\ t-1|r}}\psi\left(\frac{-\overline{r}(\beta(\vec{c})(f-F(\boldsymbol{\lambda}))+\left<\boldsymbol{\lambda},\vec{c}\right>)}{g^2}\right)\Kl_{r}\left(\overline{g}f,\frac{1}{4}\overline{g}^3F^*(\vec{c})\right)\Kl_{\infty}\left(\psi, \frac{fF^*(\vec{c})}{4r^2g^4}\right)\right|\ll_{\varepsilon,F} \hQ^{\varepsilon}|g|^{\varepsilon}\hat{T}^{1+\varepsilon}
\end{eqnarray*}
and
\begin{eqnarray*}\left|\sum_{\substack{|r|=\hat{T}\\ (g,r)=1}}\psi\left(\frac{-\overline{r}(\beta(\vec{c})(f-F(\boldsymbol{\lambda}))+\left<\boldsymbol{\lambda},\vec{c}\right>)}{g^2}\right)\Kl_{r}\left(\overline{g}f,\frac{1}{4}\overline{g}^3F^*(\vec{c})\right)\Kl_{\infty}\left(\psi, \frac{fF^*(\vec{c})}{4r^2g^4}\right)\right|\ll_{\varepsilon,F} \hQ^{\varepsilon}|g|^{\varepsilon}\hat{T}^{1+\varepsilon}.
\end{eqnarray*}
Therefore, we have reduced proving optimal strong approximation for the Morgenstern quadratic form to proving the above square-root cancellation. These would follow from the twisted Linnik\textendash Selberg square-root cancellations over function fields, that is, Conjecture~\ref{tslconj}. Indeed, we let $\alpha:=\beta(\vec{c})(f-F(\boldsymbol{\lambda}))+\left<\boldsymbol{\lambda},\vec{c}\right>$, $a:=\frac{f}{g}$, $b:=\frac{F^*(\vec{c})}{4g^3}$, and $\delta\in\{1,t-1\}$. Note that since $\vec{c}$ are such that $t-1|c_3,c_4$ (otherwise, $S_{g,r}(\vec{c})=0$), $F^*(\vec{c})\in\mathbb{F}_q[t]$, and so $b\in \mathbb{F}_q[t,g^{-1}]$. Also, recall that $\psi_r(x)=\psi\left(\frac{x\bmod r}{r}\right)$, that is, we first reduce modulo $r$, and then divide by $r$. The fact that this strong approximation implies the conjectured upper bound on the diameter of Morgenstern Ramanujan graphs for odd $q$ (see Conjecture~\ref{conj:graphs}) can be found in the introduction to the authors' paper ~\cite{SZ2019}; for more details, the reader is advised to look at Morgenstern's paper~\cite{Morgenstern}.\\
\\
\textit{Acknowledgments.} N.T. Sardari's work is supported partially by the National Science Foundation under Grant No. DMS-2015305 and is grateful to Max Planck Institute for Mathematics in Bonn and the Institute
For Advanced Study for their hospitality and financial support. 
M. Zargar was supported by SFB1085: Higher invariants at the University of Regensburg.

\bibliographystyle{alpha}
\bibliography{part2}
\end{document}